\documentclass{article}
\usepackage[utf8]{inputenc}
\usepackage[inline]{enumitem}
\usepackage{mathtools}
\usepackage{amsfonts}
\usepackage{color}
\usepackage{amsmath,amsthm,amssymb}
\usepackage{commath} 
\usepackage{todonotes}
\usepackage{cleveref}
\usepackage{shuffle}
\usepackage{url}
\usepackage[a4paper]{geometry}
\usepackage{xcolor}
\usepackage{booktabs}

\title{The Signature Kernel}
\author{Darrick Lee and Harald Oberhauser}

\newcommand{\R}{\mathbb{R}}
\newcommand{\E}{\mathbb{E}}
\newcommand{\N}{\mathbb{N}}
\newcommand{\cX}{\mathcal{X}}
\newcommand{\cF}{\mathcal{F}}

\newcommand{\cP}{\mathcal{P}}
\newcommand{\cK}{\mathcal{K}}
\newcommand{\bs}{\mathbf{s}}
\newcommand{\bw}{\mathbf{w}}
\newcommand{\bv}{\mathbf{v}}
\newcommand{\bx}{\mathbf{x}}
\newcommand{\by}{\mathbf{y}}
\newcommand{\bX}{\mathbf{X}}
\newcommand{\bY}{\mathbf{Y}}
\newcommand{\bV}{\mathbf{V}}
\newcommand{\bW}{\mathbf{W}}
\newcommand{\Hil}{H}
\newcommand{\be}{\mathbf{e}}
\newcommand{\bi}{\mathbf{i}}
\newcommand{\bj}{\mathbf{j}}

\newcommand{\bt}{\mathbf{t}}

\newcommand{\Mon}{\mathsf{Mon}}
\newcommand{\Sig}{\mathsf{Sig}}
\newcommand{\signature}{\Phi_{\Sig}}
\newcommand{\signatureTrunc}[1]{\Phi_{\Sig,:{#1}}}
\newcommand{\signatureLevel}[1]{\Phi_{\Sig,#1}}

\newcommand{\IF}{\mathsf{IF}}

\newcommand{\kernel}{\operatorname{k}}
\newcommand{\sigkernel}{\kernel_{\operatorname{\Sig}}}
\newcommand{\paths}{\cX_\text{paths}}
\newcommand{\seq}{\cX_{\text{seq}}}
\newcommand{\pathsOneVar}{\cX_\text{1-var}}

\newtheorem{example}{Example}
\newtheorem{theorem}{Theorem}

\newtheorem{proposition}{Proposition}

\theoremstyle{definition}
\newtheorem{definition}{Definition}
\newtheorem{remark}{Remark}

\begin{document}

\maketitle

\begin{abstract}
    The signature kernel is a positive definite kernel for sequential data.
    It inherits theoretical guarantees from stochastic analysis, has efficient algorithms for computation, and shows strong empirical performance.
    In this short survey paper for a forthcoming Springer handbook, we give an elementary introduction to the signature kernel and highlight these theoretical and computational properties.
\end{abstract}

\section{Introduction}
Let $\cX$ be an arbitrary set and consider
\begin{align}
\seq \coloneqq \{\bx=(\bx_0,\ldots,\bx_L) : L \ge 0,\,\bx_i \in \cX\},
\end{align}
the set of sequences of arbitrary length $L$ in the set $\cX$.
Given a kernel\footnote{Throughout, we use the term \emph{kernel} to refer to a positive definite function.} on $\cX$,
\begin{align}
\kernel : \cX \times \cX \to \R,
\end{align}
the \emph{signature kernel} transforms $\kernel$ into a kernel on $\seq$,
\begin{align}
\sigkernel : \seq \times \seq \to \R.
\end{align}
Constructing kernels for non-Euclidean domains such as $\seq$ is often challenging; in particular, even when $\cX=\R^d$ is linear, the space of sequences $\seq$ is non-linear since there is no natural addition operation for sequences of different length.
Above, we defined $\sigkernel$ on the domain of discrete-time sequences $\seq$ since this is how we typically apply it in practice when only discrete-time data is available, but the domain of $\sigkernel$ naturally extends to continuous time paths.
More generally, the signature kernel is of interest for three primary reasons.
\begin{enumerate}[label=(\roman*)]
    \item \textbf{It inherits attractive theoretical guarantees} from stochastic analysis and the properties of the classical signature.
    These manifest as universality (the ability to approximate non-linear functions), characteristicness (the ability to characterize probability measures), invariance to time-reparametrization, and convergence in the scaling limit when sequences (discrete time) approximate paths (continuous time).
 
    \item \textbf{It overcomes bottlenecks of the signature features}, both in terms of computational complexity and in terms of expressiveness since it allows us to use signatures of paths after they have been lifted to infinite-dimensional state space.
    That is, it can indirectly compute and use tensors in infininite-dimesional spaces that describe the underlying sequence.
    
    \item \textbf{It leverages modular tools from kernel learning}. Given a kernel $\kernel$ with good theoretical properties, these typically transfer to associated signature kernel $\sigkernel$, and general tools from kernel learning apply. 
    For example, so-called universality and characteristicness \cite{simon2020metrizing} of $\kernel$ are preserved for the associated signature kernel $\sigkernel$. 
    This also implies that $\sigkernel$ induces a computable metric, the so-called MMD metric, between laws of stochastic processes. 
    Further, Gaussian processes can be defined on the genuine non-Euclidean domain $\seq$ by using $\sigkernel$ as a covariance function, generic kernel quadrature constructions can be used, etc.
\end{enumerate}
A guiding theme for this survey article is that the iterated integrals, which constitute the signature, and hence the signature kernel, play the same role for sequences or paths as monomials do for vectors in $\R^d$, see Table~\ref{table:monomials}.
The key message is that this can be simultaneously a blessing and a curse.
Algebraically, polynomials have attractive properties but they quickly become computationally prohibitive in high-dimensions; analytically, trouble arises on non-compact sets, and often a more flexible class than simple monomials forms a better basis. Kernelization allows us to circumvent these drawbacks.

\section{Monomials, Polynomials, and Tensors}\label{sec:monomials}
Classical multivariate monomials and moments are convenient tools for approximating functions on $\R^d$ and respectively, representing the law of vector-valued random variables via the moment sequence.
A concise way to express the degree $m$ monomials of a vector $\bx=(\bx^1,\ldots,\bx^d) \in \R^d$ is with tensors 
\begin{align}\label{eq:monomial}
(\R^d)^{\otimes m} \ni\bx^{\otimes m} \simeq ((\bx^{1})^{m_1}\cdots (\bx^{d})^{m_d})_{m_1+\cdots m_d =m, \, m_i \in \mathbb{N}_0}.
\end{align}
We refer to the elements of $(\mathbb{R}^d)^{\otimes m}$ as tensors of degree $m$ in $d$ dimensions.

\paragraph{A Toy Example: Monomials in $\R^d$.}
Let $\cK \subset \R^d$ be compact and define %
\[
\Phi_{\Mon}(\bx) \coloneqq \left(1,\bx,\frac{\bx^{\otimes 2}}{2!},\frac{\bx^{\otimes 3}}{3!},\ldots\right) \in \prod_{m \ge 0} (\R^d)^{\otimes m} \quad \text{and} \quad \Phi_{\Mon,m}(\bx)\coloneqq \frac{\bx^{\otimes m}}{m!}.
\]
The map $\Phi_\Mon$ takes a vector $\bx = (\bx^i)_{i \in \{1,\ldots,d\}} \in \R^d$ as an input, and outputs a sequence of tensors of increasing degree. This sequence consists of a scalar $\bx^{\otimes 0} \coloneqq 1$, a vector $\bx^{\otimes 1} =\bx $, a matrix $\bx^{\otimes 2} = (\bx^{i_1}\bx^{i_2})_{i_1,i_2 \in \{1,\ldots,d\}}$, a tensor $\bx^{\otimes 3} = (\bx^{i_1}\bx^{i_2}\bx^{i_3})_{i_1,i_2,i_3 \in \{1,\ldots,d\}}$, and so on.
The co-domain of $\Phi_{\Mon}$ is a graded (in $m=0,1,2,\ldots$) linear space\footnote{Addition is defined degree-wise $(\bs_0,\bs_1,\bs_2,\ldots) + (\bt_0,\bt_1,\bt_2,\ldots) \coloneqq (\bs_0+\bs_0,\bs_1+\bt_1,\bs_2+\bt_2,\ldots)$. 
That is, we add scalars to scalars, vectors to vectors, etc.}, hence we can apply a linear functional $\bw$ to $\Phi_{\Mon}$. 
The Stone--Weierstrass theorem guarantees that for every $f \in C(\cK,\R)$ and $\epsilon >0$, there exists a $M \ge 1$ and  coefficients $\bw^0, \bw^1, \ldots, \bw^d,\bw^{1,1},\ldots,\bw^{d,d},\bw^{1,1,1},\ldots,\bw^{d,\ldots,d} \in \R$, equivalently represented\footnote{Recall that $\bw \in \oplus(\R^d)^{\otimes m}$ denotes a sequence $\bw = (\bw^I)$ where $\bw^I$ is nonzero for finitely many multi-indices $I$, whereas $\bw' \in \prod (\R^d)^{\otimes m}$ denotes a sequence $\bw' = (\bw'^I)$ where $\bw'^I$ may be nonzero for infinitely many $I$.} as $\bw \in \bigoplus(\R^d)^{\otimes m}$, such that 
\[
\sup_{\bx \in \cX} \left|f(\bx) - \underbrace{\sum_{m=0}^M \bw^{i_1,\ldots,i_m}\frac{\bx^{i_1}\cdots \bx^{i_m}}{{m!}}}_{\eqqcolon\langle \bw, \Phi_{\Mon}(\bx)\rangle} \right| < \epsilon.
\]
Hence, any choice of $\bw$ induces a function from $\cK$ to $\R$, $\left(\bx \mapsto \langle \bw, \Phi_{\Mon}(\bx)\rangle \right) \in \mathbb{R}^\cK$, which is a linear function in terms of $\Phi_\Mon$.

Another well-known result about monomials is that they can characterize laws of random variables;  in fact, convergence of the moment sequence is equivalent to classical weak convergence. 
We sum these up in the following proposition.
\begin{proposition}\label{prop:universal phi}
    Let $\cK \subset \cX = \mathbb{R}^d$ be compact. 
    Then for every $f \in C_b(\cK,\mathbb{R})$ and $\epsilon>0$ there exists some $\bw \in \bigoplus (\R^d)^{\otimes m}$ such that
    \begin{equation}\label{eq:uniform compacts}
    \sup_{\bx \in \cK}|f(\bx) - \langle \bw, \Phi_\Mon(\bx) \rangle| < \epsilon.
    \end{equation}
    Let $\bX,\bX_1,\bX_2,\ldots$ be a sequence of $\cK$-valued random variables.
    Then
    \[ \operatorname{Law}(\bX) \mapsto \mathbb{E}[\Phi_\Mon(\bX)]\coloneqq \left(1,\mathbb{E}[\bX], \frac{\mathbb{E}[\bX^{\otimes 2}]}{2!},\ldots \right)\]
    is injective and the following are equivalent:
    \begin{enumerate}
    \item\label{itm:weak convergence} $\forall f \in C_b(\cX,\mathbb{R}),\quad \lim_{n \to \infty}\mathbb{E}[f(\bX_n)] = \mathbb{E}[f(\bX)]$;
    \item\label{itm:moment convergence}  $\forall m \in \mathbb{N}, \quad\lim_{n \to \infty}\mathbb{E}[\bX^{\otimes m}_n]= \mathbb{E}[\bX^{\otimes m}]$.
    \end{enumerate}
\end{proposition}
\begin{proof}
That \eqref{itm:weak convergence} implies \eqref{itm:moment convergence} follows immediately since polynomials are bounded on compacts, hence they form a subset of $C_b(\cK,\mathbb{R})$. 
To see the reverse implication note that the compactness assumption implies tightness of the laws of $(\bX_n)$ and this allows us to use Prokhorov's Theorem. 
This reduces the proof to show that if $(\bX_n)$ converges weakly along a subsequence to a random variable $\bY$ then $\operatorname{Law}(\bY)=\operatorname{Law}(\bX)$. 
But if $\bX_n$ converges weakly to $\bY$, then $\lim_n \mathbb{E}[p(\bX_n)]=\mathbb{E}[p(\bX)]$ for any polynomial $p$ since polynomials are bounded on compacts. Hence we must have $\mathbb{E}[p(\bX)]=\mathbb{E}[p(\bY)]$ for any polynomial $p$. Finally, since polynomials are dense in $C_b(\cK,\mathbb{R})$ (by Stone--Weierstrass that we already mentioned above), the result follows. 
\end{proof}
The above shows that linear functionals of monomials, that is linear combination of the coordinates of $\Phi_\Mon$, approximate continuous functions arbitrarily well. 
Further, the expected value of the coordinates of $\Phi_{\Mon}$ characterizes the law of a random variable and convergence of these coordinates induces the classical weak convergence topology on the space of probability measures.
\begin{remark}[Universality and Characteristicness]
In the statistical learning literature, these properties of $\Phi_{\Mon}$ are known as \emph{universality} of $\Phi_\Mon$ with respect to the class of bounded and continuous functions and \emph{characteristicness} of $\Phi_\Mon$ with respect to probability measures on $\cK$. 
The general definition calls a map $\Phi: \cK \to H$ from a set $\cK$ into a Hilbert space $H$ 
\begin{enumerate*}[label=(\roman*)]
    \item \textbf{universal} with respect to a function class $\cF \subset \R^\cK$ if the space of linear functionals $\cF_0 \coloneqq \{ \langle \bw, \Phi(\cdot)\rangle: \cK \rightarrow \R \, : \, \bw \in H\}$ is dense in $\cF$;
     \item \textbf{characteristic} with respect to probability measures $\cP(\cK)$ if the expected map $\mu\mapsto \E_{X \sim \mu}[\Phi(X)]$ that maps a probability measure on $\cK$ into $\Hil$ is injective. 
\end{enumerate*}
In fact, if one uses a more general definition of characteristicness by regarding distributions, that is elements of $\cF'$, it can be shown that characteristicness and universality are equivalent. 
This makes many proofs easier, see \cite[Definition 6]{chevyrev_signature_2022} and \cite{simon2020metrizing}.
\end{remark}

The proof of~\Cref{prop:universal phi} made heavy use of the compactness assumption but this is not a deficiency of the proof. The following example shows that things really go wrong on non-compact spaces.

\paragraph{A warning: compactness and boundedness are essential.} 
Rather than a compact subset $\cK$, let's consider the entire space $\cX = \R^d$. 
In this case, moments may not exist for random variables on $\cX$ with unbounded support but the situation is even worse: even if all moments exist, they fail to characterize the law of the random variable.

\begin{figure}[!htbp]
    \centering
    \includegraphics[width=0.5\textwidth]{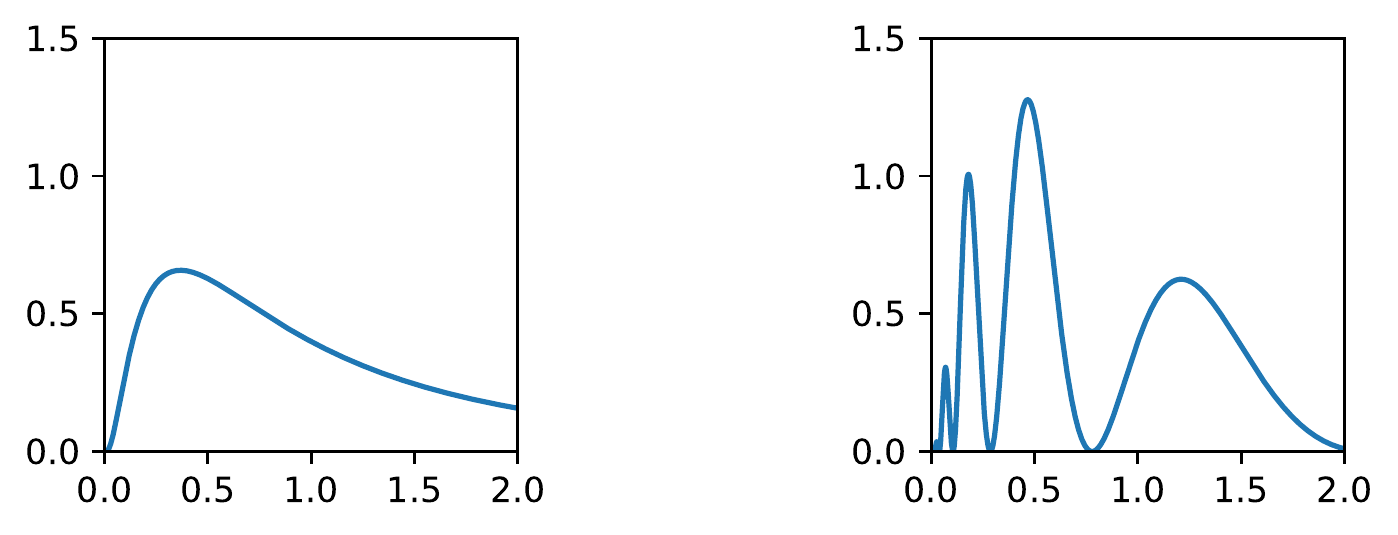}
    \caption{The probability density functions of $\bX$ (left) and $\bY$ (right).}
\end{figure}

This already happens in dimension $d=1$ as the following classical counterexample shows. Suppose $\bX$ is a standard log-normal, that is {$\exp(\bX)\sim N(0,1)$}.
There exists another real-valued random variable $\bY$ that is not log-normal\footnote{The random variable $\bY$ is explicitly specified by the density $\mathbb{P}(\bY \in dy) = f(y) (1+\sin(2 \pi \log(y))$ where $f$ denotes the density the log-normal $\bX$. In fact, \cite{heyde1963property} shows that there exist uncountably many random variables that have the same distribution as $\bX$.}, so that $\operatorname{Law}(\bX) \neq \operatorname{Law}(\bY)$, but for all $m \in \N$,
\[
\mathbb{E}[\bX^m] =\mathbb{E}[\bY^m] =  \exp\left(\frac{m^2}{2}\right).
\]
Hence, ``characteristicness'' fails on non-compact sets.
Similarly, universality, that is the uniform approximation of functions $f(\bx)$ from Equation~\eqref{eq:uniform compacts}, is hopeless since polynomials tend to $\pm\infty$ as the argument $\bx$ becomes unbounded as $\|\bx\| \rightarrow \infty$.

Polynomials are algebraically and analytically nice on compact sets, but the fundamental issue is that while their algebraic properties extend to non-compact sets (they still form a ring) we run into issues in their analysis and probabilistic properties since they grow too quickly.
We emphasize that ``tightness'' arguments are not so useful here; it is easy to find for any $\epsilon>0$ a compact set that carries $(1-\epsilon)$ of the mass of the probability measure; the issue is that the functions we care about, monomials, explode on the space domain that carries only $\epsilon$ of the mass.
A classical way to deal with this is to restrict to random variables with conditions on moment growth, but this is unsatisfactory for our purposes since such assumptions quickly become abstract and in general, it is not possible to verify them in a non-parametric inference setting as is common in machine learning tasks.
After all, things already went bad for a simple log-normal.
Finally, a closely related issue is that moments are not robust statistics, that is, if we have to estimate expectation by i.i.d.~samples from $\cX$ then small errors in these samples can lead to large estimation errors. 

We will return to all these issues but for now, we ignore this issue with non-compact sets and ask a much more basic question: What is the natural generalization of monomials from $\cX = \mathbb{R}^d$ to the space $\seq$ of sequences in $\R^d$?

\section{Iterated Integrals as Non-Commutative Monomials}
We will now go beyond the classical setting and consider monomials for sequential data. 
For simplicity we consider paths and sequence here in $\cX=\R^d$ but in Section \ref{sec:kernel learning} we discuss infinite-dimensional state spaces.
While we work with discrete sequences in $\seq$ in practice, it is often convenient to formulate our objects for continuous paths. However, we can naturally view a discrete sequence as continuous path by piecewise linear interpolation. Indeed, suppose we have a sequence $\bx = (\bx_0, \ldots, \bx_L) \in \seq$. By abuse of notation, we define the corresponding continuous path $\bx: [0,L] \rightarrow \R^d$ by
\begin{equation}\label{eq:pwl_interpolation}
    \bx(t) = (1-t) \bx_{i-1} + t\bx_i, \quad \quad t \in [i-1, i].
\end{equation}
In the continuous setting, we will consider bounded variation paths
\[
\pathsOneVar\coloneqq \{\bx\in C([0,T],\cX) \,:\,  T>0,  \,\, \|\bx \|_1 <\infty,\, \bx(0)=0\}
\]
where
\[\|\bx\|_{1}\coloneqq \sup_{\substack{0=t_0 < \cdots <t_n=T\\ n \in \mathbb{N}}} \sum_{i=1}^N|\bx(t_{i})-\bx(t_{i-1})|
\]
denotes the usual 1-variation of a path.
We restrict to paths which begin at the origin since the method that we discuss in this article (in particular, the path signature) is translation invariant.
Further, we take $\cX=\R^d$ finite-dimensional for simplicity of exposition.
There are several ways to put a topology on $\pathsOneVar$  and we take the same as in \cite{chevyrev_signature_2022}.
Finally, we note that Equation~\ref{eq:pwl_interpolation} defines an embedding
\[
    \seq \hookrightarrow \pathsOneVar.
\]

The previous article~\cite{chevyrev_primer_2016} introduced the path signature,
\[
\signature(\bx) \coloneqq \left(1,\int d\bx, \int d\bx^{\otimes 2}, \ldots \right) \in \prod_{m \ge 0} (\R^d)^{\otimes m}
\]
which maps bounded variation paths to a sequence of tensors.
The degree $m$ tensor is given as an $m$-times iterated integral, 
\[
\signatureLevel{m}(\bx)\coloneqq \int d\bx ^{\otimes m} \coloneqq \int_{0<t_1<\cdots<t_m<T} d\bx(t_1) \otimes \cdots \otimes d\bx (t_m) \in (\R^d)^{\otimes m},
\]
and by convention $(\R^d)^{\otimes 0}=\R$.
Alternatively, in coordinates, we can identify
\[\label{eq:signature}
(\R^d)^{\otimes m}\ni\int d\bx^{\otimes m}  \simeq \left(\int_{0<t_1<\cdots < t_m <T} d\bx^{i_1}(t_1) \cdots d\bx^{i_m}(t_m)\right)_{i_1,\ldots,i_m \in \{1,\ldots,d\}}
\]
where 
\[
\int_{0<t_1<\cdots < t_m <T} d\bx^{i_1}(t_1) \cdots d\bx^{i_m}(t_m)= \int_{0<t_1<\cdots < t_m <T} \dot \bx^{i_1}(t_1)  \cdots  \dot\bx^{i_m}(t_m)dt_1 \cdots dt_m
\]
is well-defined as Riemann-Stieltjes integral since bounded variation paths are almost everywhere differentiable.
For less regular paths, some care has to be taken about the definition of the iterated integrals that constitute the signature, but the same principles apply (see Section \ref{sec:applications}); for paths in infinite-dimensional state spaces the same construction applies and we discuss the case of Hilbert spaces in Section \ref{sec:kernel learning}.
Such iterated integrals can be viewed as a generalization of monomials, and here we highlight the similarities.
\begin{description}
\item{\textbf{A Graded Description.}}
One way to think about the signature map 
\[\signature: \pathsOneVar \to \prod_{m \ge 0}(\R^d)^{\otimes m}\,\quad \bx \mapsto \left(1,\int d\bx,\int d\bx^{\otimes 2},\ldots\right)\] is that it is the natural generalization of the monomial map
\[
\Phi_{\Mon}:\cX \to \prod_{m \ge 0} (\R^d)^{\otimes m}, \quad \bx \mapsto \left(1,\bx, \frac{\bx^{\otimes 2}}{2!},\ldots\right)
\]
from the domain $\cX$ to $\pathsOneVar$. 
Both describe the underlying object (a vector or a path) as a series of tensors of increasing degree.
    \item{\textbf{Signatures Generalize Monomials.}}
    The signature $\signature$ is a strict generalization of $\Phi_{\Mon}$. For a straight line, $[0,1]\ni t \mapsto t \cdot \bv$ with $\bv \in \R^d$, we have
\begin{equation}\label{eq:sig_linear_path}
\Phi_{\Sig}(t \mapsto t \cdot \bv) = \left(1,\bv,\frac{\bv^{\otimes 2}}{2!},\frac{\bv^{\otimes 3}}{3!},\ldots\right)=\Phi_{\Mon}(\bv)
\end{equation}
which follows immediately from the definition of $\signature$ since \[\int d (t\cdot \bv)^{\otimes m} = \int_{0 < t_1 < \cdots <t_m <1} \bv^{\otimes m} dt_1 \cdots dt_m = \frac{\bv^{\otimes m}}{m!}.\]
\end{description}

Despite the similarity in their definitions, there are also important differences to highlight.

\begin{description}
\item{\textbf{A Genuine Non-linear Domain.}}
Unlike $\R^d$, the domain $\pathsOneVar$ of $\signature$ is not linear: there is no natural addition of elements in $\pathsOneVar$. In particular, there is no natural way to add paths $\bx:[0,T]\to \R^d$ and $\by:[0,T'] \to \R^d$ for $T \neq T'$.
\item{\textbf{Asymmetric Tensors.}}
There is an important difference in the co-domains of the maps $\signature$ and $\Phi_{\Mon}$. 
The latter takes values in the subset of $\prod (\R^d)^{\otimes m}$ of symmetric tensors. 
For example, if we identify the degree $m=2$ tensors in $\Phi_{\Mon}$ as $d\times d$ matrices, then $\bx^{\otimes 2}$ is a symmetric matrix, since the $(i,j)$ coordinate of $\bx^{\otimes 2}$ is simply the usual product of the $i$-th and $j$-th coordinate of $\bx$, and $\bx^i \cdot \bx^j = \bx^j \cdot \bx^i$.

However, in general the tensors in $\signature$ are not symmetric. For example, by once again considering the $(i,j)$ and $(j,i)$ coordinate of $\int d\bx^{\otimes 2}$, we have in general that
\[
\int_{0<t_1<t_2<T} d\bx^i(t_1) d\bx^j(t_2) \neq \int_{0<t_1<t_2<T} d\bx^j(t_1) d\bx^i(t_2).
\]
\item{\textbf{Non-commutative Monomials.}}
Because $\Phi_{\Sig}$ results in non-symmetric tensors, the algebraic structure behind these monomials is non-commutative. Given two tensors $\bv \in (\R^d)^{\otimes m}$ and $\bw \in (\R^d)^{\otimes k}$, the tensor product $\bv \otimes \bw \in (\R^d)^{\otimes (m+k)}$ is defined coordinate-wise as
\[
    (\bv \otimes \bw)^{i_1, \ldots, i_{m+k}} \coloneqq \bv^{i_1, \ldots, i_m} \cdot \bw^{i_{m+1}, \ldots, i_{m+k}}.
\]
In general, the tensor product is \emph{non-commutative}, $\bv \otimes \bw \neq \bw \otimes \bv$. However, when we are restricted to symmetric tensor such as in $\Phi_\Mon$, the tensor product is commutative
\[
    \bx^{\otimes m} \otimes \bx^{\otimes k} = \bx^{i_1} \cdots \bx^{i_{m+k}} = \bx^{\otimes k} \otimes \bx^{\otimes m}.
\]
\end{description}
\begin{remark}
    We refer to $\Phi_{\Sig}$ as the signature (feature) map but developing a path into a series of iterated integrals is a classical construction that has been studied by many communities but known under different names such as Chen-Fliess series, Magnus expansion, time-ordered exponential, path exponential, chronological calculus etc.
    Modern origins go back to Chen's work in \cite{chen_iterated_1954,chen_integration_1957,chen_integration_1958} which can be thought as an extension of classical Volterra series \cite{schetzen_volterra_2006}. 
    The rich Lie-algebraic properties of the signature received much attention in the study of free Lie algebras \cite{reutenauer2003free} and in the 1980's, control theorists used $\Phi_{\Sig}$ extensively to linearize controlled differential equations \cite{BSMF_1981__109__3_0,brockett1976volterra}.
    Much of these results were then generalized in so-called rough path theory \cite{lyons_system_2007} to deal with ``rough'' trajectories such as the sample of paths of stochastic processes where standard integration doesn't apply. 
    A detailed history is beyond the scope of this article, but the above references point to a rich literature of disparate communities which have used (variations of) what we denote $\Phi_{\Sig}$ for their own research agenda.
\end{remark}
\section{Universality and Characteristicness of the Signature}\label{sec:universal}
In order to further establish the relationship between the monomials $\Phi_{\Mon}:\R^d \to \prod_{m \ge 0} (\R^d)^{\otimes m}$ and the signature $\signature:\pathsOneVar \to \prod_{m \ge 0}(\R^d)^{\otimes m}$, we now consider how the signature can be used to study functions and measures on $\pathsOneVar$. 
In particular, given some $\bw \in \bigoplus (\R^d)^{\otimes m}$, we define a function $\langle \bw, \signature(\cdot)\rangle: \pathsOneVar \to \R$ as before; in coordinates, 
\[\langle \bw,\signature(\bx) \rangle \coloneqq \sum_{m=0}^M \sum_{i_1,\ldots,i_m \in \{1,\ldots,d\}} \bw^{i_1,\ldots,i_m} \int_{0<t_1<\cdots <t_m<T} d\bx^{i_1}(t_1)\cdots \bx^{i_m}(t_m).\]
A direct calculation shows that the product of two such linear functionals can be expressed again as a linear functional of $\Phi_\Sig$. More formally, for $\bw_1,\bw_2\in \bigoplus (\R^d)^{\otimes m}$ there exists some $\bw_{1,2} \in \bigoplus (\R^d)^{\otimes m}$ such that
\begin{equation}\label{eq:product}
\langle \bw_1, \signature(\bx) \rangle \langle \bw_2,\signature(\bx) \rangle = \langle \bw_{1,2}, \signature(\bx) \rangle %
\end{equation}
for all $\bx \in \pathsOneVar$.
Put differently, products of (non-commutative) polynomials are again (non-commutative) polynomials.
Although the proof of this is just a direct application of integration by parts, the algebraic (so-called shuffle) structure of $\bw_{1,2}$ gives rise to many interesting algebraic questions but is beyond the scope of this short introduction.
The primary point is that \eqref{eq:product} shows that linear functionals of paths form an algebra. 
If we additionally have the point-separating property, in which
\[
\bx \mapsto \signature(\bx)
\]
is injective, then we can immediately generalize Proposition \ref{prop:universal phi} from compact subsets in $\R^d$ to compact subsets in $\pathsOneVar$.
However, in general, $\bx \mapsto \signature(\bx)$ is only injective up to so-called tree-like equivalence. 
We refer to \cite{hambly_uniqueness_2010} for details on tree-like equivalence.\medskip

We can work around this tree-like equivalence in the next Proposition by introducing a time coordinate: rather than using $t \mapsto \bx(t) \in \R^d$, we consider 
\begin{align}\label{eq:time_parametrization}
t \mapsto \bar \bx (t) \coloneqq (t,\bx(t))\in \R^{d+1}.
\end{align}

When we consider such time-parametrized paths, the signature $\signature$ is injective, and we obtain the following generalization of Proposition \ref{prop:universal phi}.

\begin{proposition}\label{prop:universal sig}
Let $\cK \subset \pathsOneVar$ be a compact set and denote with $\bar \cK=\{\bar \bx: \bx \in \cK\}$ the set of paths in $\cK$ augmented with a time coordinate from Equation~\ref{eq:time_parametrization}. %
Then for every $f\in C_b(\bar \cK,\R)$ and $\epsilon>0$ there exists a $\bw \in \bigoplus (\R^{d+1})^{\otimes m}$ such that
\[
\sup_{\bx \in \cK} | f(\bar \bx) - \langle \bw, \signature(\bar \bx) \rangle | < \epsilon.
\]
Let $\bX,\bX_1,\bX_2,\ldots$ be a sequence of $\cK$-valued random variables, that is, each $\bX_i=(\bX_i(t))_{t \in [0,T_i]}$ is a stochastic process.
Then the expected signature map
\[
\operatorname{Law}(\bar \bX) \mapsto \mathbb{E}[\signature(\bar \bX)]
\]
is injective and 
\begin{enumerate}
    \item\label{itm:sig weak convergence} $\forall f \in C_b(\bar \cK,\mathbb{R})\quad \lim_{n \to \infty}\mathbb{E}[f( \bar \bX_n)] = \mathbb{E}[f(\bar \bX)]$,
    \item\label{itm:sig moment convergence} $\forall m \in \mathbb{N}, \quad\lim_{n \to \infty}\mathbb{E}[\int d \bar \bX^{\otimes m}_n]= \mathbb{E}[\int d\bar \bX^{\otimes m}]$.
    \end{enumerate}
\end{proposition}

This proposition is a direct generalization of Proposition \ref{prop:universal phi}; in fact, it can be proved in the same way.
The only difference is that one needs to show $\bx \mapsto \signature(\bar\bx)$ is injective.
However, this is not too difficult, see for example \cite{litter}.
Without the augmentation with a time coordinate, the same statement holds when $\cK$ is replaced by set of equivalence classes of paths, so-called tree-like equivalence\footnote{Many choices of topologies for paths and unparametrized path are possible; for a detailed study for unparametrized paths, see \cite{cass_topologies_2022}.}.
With all of these properties in mind, we can see that the signature map behaves very similarly to classical monomials, and we summarize these results in Table \ref{table:monomials}.

\renewcommand{\arraystretch}{1.4}
\begin{table}
\centering
    \begin{tabular}{ccc} \toprule
      & Monomial Map & Path Signature \\ \midrule
    Domain & $\cX = \R^d$ & $\pathsOneVar$ \\ 
    Co-domain & $\prod_{m \ge 0} (\R^d)^{\otimes m}$ & $\prod_{m \ge 0} (\R^d)^{\otimes m}$ \\ 
    Map & $\Phi_{\Mon}(\bx) = \left(\frac{\bx^{\otimes m}}{m!}\right)_{m \geq 0}$ & $\signature(\bx) = \left( \int d\bx^{\otimes m}\right)_{m \geq 0}$ \\ 
    Type of Tensors & symmetric & non-symmetric \\ 
    Continuity & w.r.t. Euclidean metric & w.r.t. 1-variation metric \\ 
    Factorial Decay & $\|\Phi_{\Mon, m}(\bx)\| \leq \frac{\|\bx\|}{m!}$ & $\|\signatureLevel{m}(\bx)\| \leq \frac{\|\bx\|_{1}}{m!}$ \\ 
    Scaling & $\Phi_{\Mon, m}(\lambda \bx) = \lambda^m \Phi_{\Mon, m}(\bx)$ & $\signatureLevel{m}(\lambda \bx) = \lambda^m \signatureLevel{m}(\bx)$ \\ 
    \midrule
     & compact $\cK \subset \cX$ & compact $\cK \subset \pathsOneVar$\\
     Universality & w.r.t. $C(\bar \cK,\R)$ & w.r.t. $C(\cK, \R)$ \\ 
     Characteristicness & w.r.t. $\cP(\bar \cK)$ & w.r.t. $\cP(\cK)$ \\
    \bottomrule
    \end{tabular}
    \caption{Comparison between classical monomials and the signature.}
    \label{table:monomials}
\end{table}
\section{The Good, the Bad, and the Ugly} \label{sec:good_bad_ugly}
While classical monomials have many attractive properties, monomials also come with some downsides that are inherited by the ``non-commutative monomials'', i.e.~the iterated integrals, that constitute $\Phi_{\Sig}$.
\begin{description}
\item[Good.] 
Monomials are building blocks to construct more complicated functions and have a well-understood algebraic structure.
For example, linear combinations of monomials (i.e.~polynomials) are closed under multiplication and form a ring, and any continuous function with compact support can be uniformly approximated by polynomials.
\item[Bad.] 
\begin{itemize}
    \item 
Monomials and hence also polynomials fail miserably on non-compact domains; they are great at providing local approximations but explode on unbounded domains. 
\item
On unbounded domains, the moment sequence may not exist. 
But even if it does, that is all moments exist, the moment sequence does in general not uniquely characterize the law of random variable. 
\item
Computational complexity: The tensor $\bx^{\otimes m} \in (\R^d)^{\otimes m}$ has $\binom{d+m-1}{m} =O(d^m)$ coordinates so that the computational cost becomes quickly prohibitive for moderate $d$ which makes regression $f(\bx) \approx \langle \bw, \Phi_{\Mon}(\bx) \rangle$ and working with moments $\mathbb{E}[\bX^{\otimes m}]$ too expensive.
\item Non-robustness: Small changes in the underlying probability measure can lead to large changes in the moment sequence. This becomes an issue for estimation; e.g.~one outlier among the sample can completely destroy moment estimates.
\end{itemize}
\item[Ugly.] 
While monomials can theoretically approximate functions and capture probability distributions, they can be a poor choice in practice.
In data science, arguably the most successful method is to learn the approximation from the data itself. 
In practice, this means providing a large parameterized family of functions and optimize over the parameter when applied to data and standard` monomials are rarely the optimal choice.
\end{description}
Because $\Phi_{\Sig}$ generalizes $\Phi_{\Mon}$ it shares the same issues.
However, all these issues are well-known and modern statistics and machine learning have developed tools to address these shortcomings. 
The remainder of this article will focus on the following two tools and develop them for $\Phi_{\Sig}$.
\begin{itemize}
\item \textbf{Kernelization} allows us to simultaneously consider a rich set of non-linearities while avoiding the combinatorial explosion in the computation of monomials. This is done by sidestepping the computation of individual monomials by obtaining efficient algorithms to directly compute the inner product between monomials.
\item \textbf{Robust Statistics} was developed in the 1980's \cite{huber_robust_2011} to provide both a formalism to quantify the robustness of a given statistic and tools to make a given statistic robust in this precise mathematical sense.
We use this to simultaneously resolve the issues with $\Phi_{\Mon}$ and $\Phi_{\Sig}$ on non-compact sets and deal with their general non-robustness.
\end{itemize}
The problems of extending kernelization and robustness methods to the path signature $\signature$ address independent issues although they can be combined. 
We discuss kernelization in Sections \ref{sec:kernel learning} and \ref{sec:signature kernel}; robustification is discussed in Section \ref{sec:robust}.

\section{Kernel Learning}\label{sec:kernel learning}
We briefly recall the main ideas behind kernelization and we refer to \cite{scholkopf2018learning,cucker2002mathematical} for a detailed introduction to kernel learning.

\paragraph{Kernelization in a Nutshell.}%
Given a set $\cK$, we want a feature map $\varphi: \cK \to \Hil$ that injects elements of $\cK$ into a Hilbert space $\Hil$ such that $\varphi$ provides sufficient non-linearities so that firstly, linear functionals of $\varphi$, $\bx \mapsto \langle \bw, \varphi(\bx) \rangle$, are expressive enough to approximate a rich set of functions and, secondly, such that $\E[\varphi(\bX)]$ captures the law of a $\cK$-valued random variable $\bX$. 
The examples we encountered are $\cK \subset \cX = \R^d$ with $\varphi=\Phi_\Mon$ as the feature map and $\cK \subset \pathsOneVar$ with $\varphi=\signature$ as the feature map.
Even for $\cK=\R$ there are many choices of $\varphi$ and~$\Hil$ so let's take a step back and simply ask what could be a good candidate for the feature space $\Hil$. 

We want $\Hil$ to be linear, large, and canonical with respect to the underlying data space $\cK$. 
A canonical linear space that we can associate with any set $\cK$ is the the space of real-valued functions $\R^\cK$ on $\cK$.
Hence, we take $\Hil$ to be subset of $\R^\cK$, that is $\Hil$ will be a space of real-valued functions on $\cK$.
Consequently the feature map $\varphi$ will be function-valued, 
\begin{align}\label{eq:kernel feature}
\cX\ni\bx \mapsto \varphi(\bx) \coloneqq \kernel(\bx,\cdot) \in \Hil\subset \R^{\cK}.
\end{align}
While this has the potential to provide a rich set of ``non-linearities'' by embedding $\cK$ in the (generically infinite-dimensional) space $\Hil$, the drawback is that learning algorithms require the evaluation of $\varphi(\bx)=\kernel(\bx,\cdot)$ which makes the direct use of \eqref{eq:kernel feature} infeasible. 
Kernel learning rests on three observations.
\begin{enumerate}
\item If we assume that $\kernel$ is positive definite, we can identify $\Hil$ as the Reproducing Kernel Hilbert Space (RKHS) induced by $\kernel$, that is $\Hil$ is the closure of 
\[
\left\{\bx \mapsto \sum_{i=1}^n \kernel(\bx_i,\bx)\, :\, \bx_i \in \cK,\,  n \in \mathbb{N}\right\} \subset \R^\cK.
\]
This induces the reproducing property, where $\langle f, \kernel(\bx,\cdot) \rangle_\Hil = f(\bx)$ for $f \in \Hil$. 
As a consequence of the reproducing property, we have $\langle \varphi(\bx), \varphi(\by) \rangle_\Hil = \kernel(\bx,\by) $.
Hence, we may think of $\kernel$ as a ``nonlinear inner product''.
\item The evaluation of $\kernel(\bx,\by)$ can often be made computationally cheap even though it is the inner product of elements $\varphi(\bx)$ and $\varphi(\by)$ in (the often infinite-dimensional) $\Hil$. For example, arguably the most popular kernel is the radial basis function (RBF) kernel, $\kernel(\bx,\by)=\exp(-\sigma |\bx-\by|^2)$, which is the inner product of an infinite-dimensional map $\varphi(\bx)$. 
That an inner product of feature maps $\varphi$ is computationally cheap, is often referred to as a ``kernel trick''.
\item Many learning algorithms can be reformulated so that they do not require the evaluation of $\varphi(\bx)$ but only the computation of the inner product $\kernel(\bx,\by)=\langle\varphi(\bx),\varphi(\by)\rangle$.
\end{enumerate}
To summarize, for a given dataset $\mathcal{D}=\{\bx_1,\ldots,\bx_N\} \subset \cX$, kernelization allows us to work with highly non-linear feature maps $\varphi(\bx)$ taking values in an infinite dimensional space $\Hil$ by only evaluating the Gram matrix $(\kernel(\bx_i,\bx_j))_{i,j \in N}$.

\begin{remark}
Kernelization leverages duality for a complexity trade-off: while we do not have to worry about the dimension of $\Hil$ in the computation of kernels, we must compute the entire Gram martrix which scales quadratically in the number of samples $N$.
This is a direct consequence of the duality that underlies kernelization.
\end{remark}
To gain some intuition we revisit and kernelize the classical monomial feature map from Section \ref{sec:monomials},
\[
\Phi_\Mon: \R^d \to \prod_{m \ge 0} (\R^d)^{\otimes m},\bx \mapsto \left(1,\bx, \frac{\bx^{\otimes 2}}{2},\ldots\right).
\]
Further, we denote with 
\[
\Phi_{\Mon, :M} = \R^d \to \prod_{m \ge 0} (\R^d)^{\otimes m}, \bx \mapsto \left(1,\bx, \frac{\bx^{\otimes 2}}{2},\ldots,\frac{\bx^{\otimes M}}{M!},0,0,\ldots\right)
\]
the same map, but truncated at level $M$; here $0 \in (\R^d)^{\otimes M}$ denotes the degree $M$ tensor that has as entries all zeroes.
The co-domain of $\Phi_{\Mon}$ and $\Phi_{\Mon, :M}$ is the same, but the latter has only finitely many non-zero entries. 
We refer to $\Phi_{\Mon, :M}$ as truncated at $M$. 
\paragraph{Inner Product of Tensors.}
There is a natural inner product on $\R^d$ by setting
\[
\langle \bs_m,\bt_m \rangle_{m}\coloneqq \sum_{i_1, \ldots, i_m =1}^d \bs^{i_1, \ldots, i_m}_m \cdot \bt^{i_1, \ldots, i_m}_m,
\]
where $\bs_m, \bt_m \in (\R^d)^{\otimes m}$. This inner product can be extended to the subset of $\prod (\R^d)^{\otimes m}$ as
\[
\langle \bs, \bt \rangle \coloneqq \sum_{m=0}^\infty \langle \bs_m, \bt_m\rangle_{m},
\]
for which the above sum converges;
where $\bs = (\bs_0, \bs_1, \ldots), \bt = (\bt_0, \bt_1, \ldots) \in \prod (\R^d)^{\otimes m}$. Hence, if we denote by $\be_i \in \R^d$ the vector that is all zeros except with a $1$ at entry $i$, then an orthogonal basis of $\prod (\R^d)^{\otimes m}$ is given by
\[
\{\be_{i_1} \otimes \cdots \otimes \be_{i_m} \, :\, i_1,\ldots,i_m \in \{1,\ldots,d\}^m,\,  m \in \mathbb{N}\}.
\]
The following is an elementary but key property that we use throughout
\begin{align}\label{eq:product of tensors}
\langle \bx_1 \otimes \cdots \otimes \bx_m, \by_1 \otimes \cdots \otimes \by_m \rangle = \langle \bx_1,\by_1\rangle \cdots \langle \bx_m,\by_m \rangle \text{ for all }\bx_1,\ldots,\bx_m,\by_1,\ldots,\by_m \in \R^d.
\end{align}
This already hints at how we can reduce the computational complexity. A naive evaluation of the left-hand side requires us to store the $d^m$ coordinates of the tensors $\bx_1 \otimes \cdots \otimes \bx_m$ and $\by_1 \otimes \cdots \otimes \by_m $ and then compute the inner product, hence requires $O(d^m)$ memory. 
In contrast, evaluating the right hand side requires just $O(dm)$ memory.

\paragraph{Toy Example 1: Avoiding the Computational Complexity.}
Let $\cX \subset \R^d$ and define
\[\kernel_{\Mon, :M}^0:\cX \times \cX \to \R, \quad \kernel_\Mon^0(\bx,\by) = \langle \Phi_{\Mon, :M}(\bx), \Phi_{\Mon, :M}(\by)\rangle\]
We use the $0$ to denote that it is the ``0''-th version since we going to improve it in the next example. 
But for now note that using \eqref{eq:product of tensors} we get 
\begin{align}\label{eq:kernel trick monomial 1}
\kernel_{\Mon, :M}^0(\bx,\by)=\sum_{m=0}^M \frac{\langle \bx^{\otimes m}, \by^{\otimes m}\rangle}{(m!)^2} = \sum_{m=0}^M \frac{\langle \bx, \by\rangle^m}{(m!)^2}.
\end{align}
which already shows that the computational cost is linear in $d$. 
We can make this even more efficient, by recalling \emph{Horner's scheme}: the formula \eqref{eq:kernel trick monomial 1} is a polynomial in $\langle \bx, \by\rangle$ which can be efficiently computed as
\begin{align}\label{eq:horner}
    k_{\Mon, :M}^0(\bx,\by) = 1 + \langle \bx, \by\rangle \left( 1 + \frac{\langle \bx, \by\rangle}{2^2}\left(1 + \ldots \frac{\langle \bx, \by\rangle}{(M-1)^2}\left( 1 + \frac{\langle \bx,\by\rangle}{M^2}\right)\right)\right),
\end{align}
This ``Horner-type'' evaluation has three advantages: firstly, it  requires only $M$ additions and $M$ multiplications in contrast to the naive evaluation of \eqref{eq:kernel trick monomial 1}; secondly, it is numerically stable, again in contrast to \eqref{eq:kernel trick monomial 1} which adds powers of scalars which live on different scales and makes floating point arithmetic hard; and thirdly, it allows us to replace the inner product $\langle \bx, \by \rangle$ with kernel $\kernel(\bx,\by)$ which we describe in the toy example below.

\paragraph{Toy Example 2: Composing Non-Linearities.}
We now have an efficient way to evaluate $\kernel_{\Mon}^0$ but it is well-known that this kernel does not perform well in practice \cite{scholkopf_learning_2002-1}.
The reason is the aforementioned fact, that although monomials have nice  theoretical properties, they suffer from a lack of expressivity; in practice they are outperformed by other kernels on $\R^d$ that are inner products of other non-linearities, e.g.~the RBF kernel.

A natural way to address this is to first apply a non-linear map \[\varphi: \R^d \to \Hil\] and only subsequently compute the monomials in $\Hil$, that is we consider 
\begin{align}\label{eq:monomial nonlinear}
\Phi_{\Mon, :M}\left( \varphi(\bx) \right)\equiv \left(1,\varphi(\bx),\frac{\varphi(\bx)^{\otimes 2}}{2!}, \ldots, \frac{\varphi(\bx)^{\otimes M}}{M!},0,0,\ldots\right).
\end{align}
This preserves the algebraic structure but additionally allows for a far more expressive feature map since $\varphi$ adds (infinitely-many) additional non-linearities via $\varphi$.
If we take this underlying map to be the identity, $\varphi = \operatorname{id}$, then \eqref{eq:monomial nonlinear} coincides with \eqref{eq:monomial}; but in general we can use arbitrary non-linearities $\varphi$.
Unfortunately, a direct use of the map $\Phi_{\Mon}\circ \varphi$ is impossible in practice since we cannot directly evaluate $\varphi$ when $\Hil$ is infinite-dimensional, even for $M=2$. 
The key remark now is that if we take $\varphi(\bx) = \kernel(\bx,\cdot)$ to be the feature map of a kernel $\kernel:\R^d \times \R^d \to \R$ that has $\Hil$ as RKHS, then \emph{the inner product }
\begin{align}
    \kernel_{\Mon, :M}(\bx,\by) \coloneqq \langle \Phi_{\Mon, :M} \circ \varphi(\bx), \Phi_{\Mon, :M} \circ \varphi(\by) \rangle = \langle \Phi_{\Mon, :M} \circ \kernel(\bx,\cdot), \Phi_{\Mon, :M} \circ \kernel(\by,\cdot) \rangle 
\end{align}
\emph{can be computed exactly}! 
In fact, it just amounts to replacing the Euclidean inner product $\langle \bx, \by \rangle $ in the formula \eqref{eq:kernel trick monomial 1} and ~\eqref{eq:horner} with $\kernel(\bx,\by)$. 
Note the kernel $\kernel_{\Mon, :M}$ depends on $M$ and $\varphi$; the optimal selection of these is important and we discuss this further in Section \ref{sec:hyperparameter}.

While the calculation itself is elementary, it is remarkable that this can be done despite the genuine infinite-dimensional co-domain of $\Phi_{\Mon, :M} \circ \varphi$; such a trade-off of losing direct access to the feature map but being able to compute an inner product is often called a kernel trick.
Thus the kernel $\kernel_{\Mon, :M}$ inherits the nice algebraic properties of monomials and also the power to access an infinite-dimensional $\Hil$ via $\kernel$.

\section{The Signature Kernel}\label{sec:signature kernel}
We now apply the same kernelization approach as in Section \ref{sec:kernel learning} to the signature feature map $\signature$ in order to obtain a computable kernel equipped with further non-linearities.
We begin by considering a kernel
\begin{align}
\kernel : \R^d \times \R^d \to \R, 
\end{align}
on the state space of our paths; henceforth we refer to $\kernel$ as the \emph{static kernel}. Instead of lifting elements of $\R^d$ into the RKHS $\Hil$ of $\kernel$, we use it to map the path $\bx$,
\[
t \mapsto \bx(t) \in \R^d,
\]
into a path $\kernel_\bx$,
\[
t \mapsto \kernel_\bx(t) \coloneqq \kernel(\bx(t),\cdot) \in \Hil,
\]
that now evolves in $\Hil$.
For a generic static kernel $\kernel:\cX \times \cX \to \R$, its RKHS $\Hil$ will be infinite-dimensional, hence the path $\kernel_\bx:t\mapsto \kernel_\bx(t)$ evolves in an infinite-dimensional space. 
Our aim is to use the signature of the path $\kernel_\bx$, 
\begin{align}\label{eq:signature of kernel lift}
\signature(\kernel_\bx)= \left(1, \int d \kernel_\bx, \int d \kernel_\bx^{\otimes 2},\ldots \right),
\end{align}
and its truncated version
\begin{align}
    \signatureTrunc{M}(\kernel_\bx) \coloneqq \left(1, \int d \kernel_\bx, \int d \kernel_\bx^{\otimes 2},\ldots, \int d \kernel_\bx^{\otimes M},0,0,\ldots\right)
\end{align}
as a way to represent the path $\bx$. Once again, the intuition is that \[\signature(\kernel_\bx) \in \prod_{m \ge 0} \Hil^{\otimes m}\] is much more expressive than $\signature(\bx)$ since it computes iterated integrals of the path after the non-linearities of $\kernel$ have been applied so we will have a much richer set of coordinates (of iterated integrals) to choose from when $\kernel$ is non-linear and the RKHS of $\kernel$ is very large.

Again the bad news is, of course, that even for $m=2$ storing $\int d \kernel_\bx^{\otimes 2} \in \Hil^{\otimes 2}$ exactly is in general infeasible for the interesting case when $\Hil$ is infinite-dimensional.
The main insight in \cite{kiraly_kernels_2019} is that, perhaps surprisingly, the inner product 
\begin{align}\label{eq:sigkernel}
\kernel_{\Sig, :M}(\bx,\by) \coloneqq \langle \signatureTrunc{M}(\kernel_\bx), \signatureTrunc{M}(\kernel_\by) \rangle
\end{align}
can be computed exactly when $M<\infty$ for piecewise linear paths $\bx,\by$ and, perhaps less surprisingly, inherits many of the attractive properties of the signature features.

\begin{theorem}\label{thm:sigkernel}
    Let $\kernel:\R^d\times \R^d \to \R$ be a kernel and let $M \in \mathbb{N}$ and let
    \begin{align}
        \kernel_{\Sig, :M}: \pathsOneVar \times \pathsOneVar \to \R
    \end{align}
    be as defined in \eqref{eq:sigkernel}. 
    Then $\kernel_{\Sig, :M}(\bx,\by)$ can be evaluated exactly for piecewise linear paths $\bx=(\bx_1,\ldots,\bx_{L_\bx})$, $\by=(\by_1,\ldots,\by_{L_\by})$ in $O(M^3L^2)$ computational steps where $L=\max(L_\bx,L_\by)$.
\end{theorem}

Even in the case of $M = \infty$, when $\sigkernel(\bx,\by)\equiv \langle \signatureTrunc{\infty}(\bx), \signatureTrunc{\infty}(\by) \rangle \equiv \langle \signature(\bx), \signature(\by) \rangle $, good approximations of $\sigkernel(\bx,\by)$ exist for piecewise linear paths which we discuss further in the next section.

\subsection{Approximations}
The above Theorem \ref{thm:sigkernel} provides a powerful method to compute signature kernels and allows for many applications.
However, besides the exact evaluation of $\kernel_{\Sig, :M}$, approximations are important for various reasons.
\begin{description}
    \item[Gram Matrix Approximations.] In a machine learning context, the computational bottleneck is not just the evaluation of a single $\kernel_{\Sig, :M}(\bx,\by)$ but the computation of a Gram matrix 
    \[
    G = (\kernel_{\Sig, :M}(\bx,\by))_{\bx,\by \in \mathcal{D}}
    \]
    where $\mathcal{D} \subset \seq$ is the finite set of data we are given.
    However, naively evaluating each of the $|\mathcal{D}|^2$ entries of this matrix is typically not feasible since the datasets can be very large.
    General kernel learning approaches such as the Nystr\"om-method can be applied and replace the original Gram matrix by an approximation with a low-rank matrix which can have linear complexity in $|\mathcal{D}|$.
    Similarly, in a Gaussian process (GP) context, one not only needs to approximate $G$ but also calculate its inverse as part of the Bayesian update rule. 
    This scales even worse than quadratic in the number of data points, and again some generic GP methods have been developed to deal with this.

    \item[Untruncated Signature Kernel Approximations.]
    We regard the truncation level $M \in \N_{\infty}$ as a hyperparameter that should be tuned by the data. While Theorem \ref{thm:sigkernel} gives an exact algorithm for finite $M$ (in practice, the algoritm of Theorem \ref{thm:sigkernel} is feasible for $M<10$), one would like to know how well we can approximate the signature kernel with a large $M$ or even $M=\infty$. 
    
    \item[Discrete to Continuous Time Approximation.]
    Sometimes the data is genuinely a discrete-time sequence, but sometimes the underlying object is a continuous-time path of which we are only given discrete-time samples. 
    The signature kernel is well-defined for continuous-time paths, hence it can be useful to quantify the time-discretization error.
   These limiting paths are typically not of bounded variation, for example Brownian motion is only of finite $p$-variation for $p>2$, so some care has to be taken.
    
\end{description}
All three of the above points address different issues, and the following discusses them in more detail. 

\paragraph{Gram-Matrix approximations.}
A general issue in kernel learning is that the computation of the Gram matrix scales quadratically in the number $N=|\mathcal{D}|$ of data points $\mathcal{D}$. 
In particular, each data point $\bx \in \mathcal{D}$ corresponds to a sequence, hence we would require $N^2$ evaluations of the signature kernel. 
Even if a single evaluation of $\kernel_{\Sig, :M}$ is cheap, this quadratic complexity becomes quickly prohibitive. 
A generic method that applies to any kernel $\kernel$ is the Nystr\"om approximation, which in its simplest form, approximates a general kernel Gram matrix 
\[G=(\kernel(\bx_i,\by_j))_{i,j=1,\ldots,N}
\]
by sampling uniformly at random a set $J$ a subset of $\{1,\ldots,n\}$ of size $r$ and define
\[
\tilde G = C W^{-1} C^\top
\]
where $C= (\kernel(\bx_i,\bx_j))_{i=1,\ldots,N,\, j \in J}$ and $W=(\kernel(\bx_i,\bx_j))_{i,j \in J}$. 
Note that $\tilde G$ is a random matrix of rank $r$, hence the computational cost scales as $r^2$ which can be huge reduction if $r\ll N$. 
Besides uniform subsampling of columns, more sophisticated Nystr\"om approximations schemes are available \cite{drineas2005nystrom}.
All these generic method can be applied to the signature kernel: selecting $J$ amounts to selecting a subset of the sequences in our dataset. 
However, this only addresses the complexity in the number $N$ of sequences, but not the other computational bottleneck which is the quadratic complexity in sequence length $L$. 
A simple way to address the latter is to subsample each sequence, which again, can work but is a brute force approach.
A more elaborate option is to also apply the Nystr\"om approximation at the level of time-points; \cite{kiraly_kernels_2019}[Algorithm 5] shows that these two Nystr\"om approximations can be done simultaneously: informally, one does not only reweigh the importance of sequences but also of time-points within these sequences.

Another situation where the cost of the matrix calculations become quickly too prohibitive is for Gaussian processes. 
We discuss Gaussian processes with signature kernels in Section \ref{sec:applications} and again generic GP approximation methods apply. 
However, we can further reduce the complexity of our algorithms by exploiting the specific structure of the signature kernel. 
This can be achieved by selecting special tensors~\cite{toth_bayesian_2020} or by using diagonalization arguments~\cite{lemercier2021siggpde}.

\paragraph{Truncation Approximation.} 
The signature tensors decay factorially (see Table \ref{table:monomials}), and this allows us to quantify the difference between truncated and untruncated signatures as follows.
\begin{proposition}
    Let $\signatureTrunc{\infty}=\signature$ be the untruncated signature map and let $M> 0$. 
    Then, for any $\bx \in \pathsOneVar$, we have 
    \begin{align}\label{eq:truncate error}
        \|\signatureTrunc{\infty}(\kernel_{\bx}) - \signatureTrunc{M}(\kernel_{\bx})\| \leq \frac{e^{\|\kernel_{\bx}\|_1} \|\kernel_{\bx}\|^{M+1}_1}{(M+1)!}.
    \end{align}
\end{proposition}
A proof can be found in~\cite[Lemma 46]{chevyrev_signature_2022}.
Despite the nice factorial decay, we emphasize that the variation norm of the sequence appears. 
Furthermore, we can obtain an analogous bound on the truncation error for the kernel by applying Cauchy-Schwarz.
\begin{proposition}
    Let $\kernel_{\Sig,:\infty}$ be the untruncated signature kernel and $\kernel_{\Sig,:M}$ be the signature kernel truncated at $M> 0$. Then, for any $\bx, \by \in \pathsOneVar$, we have
    \begin{align*}
        |\kernel_{\Sig,: \infty}(\bx, \by) - \kernel_{\Sig,:M}(\bx, \by)| \leq \frac{e^{C} C^{M+1}}{(M+1)!},
    \end{align*}
    where $C = \max\{ \|\kernel_{\bx}\|_1, \|\kernel_{\by}\|_1\}$.
\end{proposition}

The above quantifies that by choosing a sufficiently high truncation level $M$, we obtain a good approximation to $\kernel_{\Sig,\infty}$. 
In practice, $M=10$ is a feasible computation and since $10! \approx 10^6$, the contribution of the tail in the sum is negligible. 
However, we draw attention to the fact that the 1-variation of the sequence appears as a constant. 
Hence, the level $M$ at which the factorial decay kicks in is determined by the path length. The situation is analogous to computing the matrix exponential where the matrix norm determines this truncation level, see Figure 1 in \cite{moler2003nineteen}.

Instead of using an exact algorithm for finite $M$ and regarding $\kernel_{\Sig,:M}$ as an approximation to $\kernel_{\Sig,:\infty}$, one can also directly approximate $\kernel_{\Sig,:\infty}$ by numerically solving a PDE~\cite{salvi_signature_2021-1}.
In this case, the approximation error to $\kernel_{\Sig,:\infty}$ is due to the discretization of the PDE solver. In particular, the error is quantified in terms of the step size; see \cite{salvi_signature_2021-1} for further details.

\paragraph{Discrete Time Approximation.}
While we have formulated the kernel in terms of continuous paths, this immediately provides a kernel for discrete sequences via the piecewise linear interpolation from Equation~\ref{eq:pwl_interpolation}. Given a discrete sequence $\bx = (\bx_0, \ldots, \bx_L) \in \seq$ and the static kernel, we consider the induced continuous path $\kernel_{\bx}:[0,L] \to H$, defined by
\begin{equation} \label{eq:pwl_kernel_approximation}
    \kernel_{\bx}(t) = (1-t)\kernel_{\bx_{i-1}} + t\kernel_{\bx_i}, \quad \quad t\in [i-1, 1].
\end{equation}
Then, we can use the algorithms from~\cite{kiraly_kernels_2019} with the complexity bound from Theorem \ref{thm:sigkernel} to efficiently compute signature kernels for discrete data.

We can also consider sequences as discretizations of continuous paths. Given a continuous path $\bx: [0,T] \rightarrow \R^d$, we consider the discretization $\bx^\pi \in \seq$ with respect to a partition $\pi = (t_i)_{i=0}^L$ of the interval $[0,T]$, where $0 = t_0 < \ldots < t_L = T$. Applying the above interpolation to $\bx^\pi$, we obtain an approximate continuous path $\kernel_{\bx}^\pi$ induced by $\bx^\pi$. Thus, we obtain an approximation $\signatureTrunc{M}(\kernel_\bx^\pi)$, which converges to the signature $\signature(\kernel_\bx)$, as $M\to \infty$ and as the mesh of the partition $\pi$, $\sup_{t_i \in \pi} |t_i-t_{i-1}|$, vanishes. 
The convergence rate is given by the following result; further estimates can be found in \cite{kiraly_kernels_2019}.

\begin{proposition}[\cite{kiraly_kernels_2019}]
    Let $\bx \in \pathsOneVar$ with domain $[0,T]$ and suppose $\pi = (t_i)_{i=0}^L$ be a partition of $[0,T]$. Then,
    \begin{equation}
        \|\signatureTrunc{\infty}(\kernel_{\bx}) - \signatureTrunc{M}(\kernel_{\bx}^\pi)\| \leq e^{\|\kernel_{\bx}\|_1}\|\kernel_{\bx}\|_1  \cdot \max_{i=1, \ldots, L} \|\kernel_{\bx}|_{[t_{i-1}, t_i]}\|_1.
    \end{equation}
\end{proposition}
\subsection{Algorithms}\label{sec:algo}
A direct evaluation of the signature kernel $\sigkernel:\seq \times \seq \to \R$ as an inner product is only feasible if the underlying static kernel $\kernel: \cX \times \cX \to \R$ has a low-dimensional RKHS $\Hil$. If $\cX=\R^d$ and $\kernel$ is the Euclidan inner product, then $\sigkernel$ is just the inner product of plain-vanilla signatures and even this becomes quickly costly.
However, the intuition is that by first lifting the path into a path that evolves in a high- or infinite-dimensional RKHS $\Hil$, we gain much more expressive power. 
Indeed, this is also what the empirical data shows, see e.g.~\cite{toth_bayesian_2020}, namely that using popular choices such as RBF, Matern, or Laplace for the static kernel results in a signature kernel $\sigkernel$ that outperforms the ``trivial'' signature kernel induced by the inner product kernel on $\R^d$. 
The computation for nontrivial kernels cannot be done directly since the RKHS $\Hil$ of a generic kernel $\kernel$ is infinite-dimensional and we can't compute signature of paths evolving in general RKHS $\Hil$. 
But there are at least two different algorithms

\begin{description}
    \item[Dynamic Programming.]
    The idea is to mimic the argument for the monomial kernel $\kernel_\Mon$ in $\R^d$, given in Toy Example 1 and Toy Example 2 in Section \ref{sec:kernel learning}.
    Concretely, given $\bx=(\bx_0,\ldots,\bx_L) \in \seq$, we consider the signature of $\kernel_\bx$, which is the sequence $\bx$ lifted to a sequence in $\Hil$ defined in Equation~\ref{eq:pwl_kernel_approximation}.
    First, the signature of one linear segment (see Equation~\ref{eq:sig_linear_path}) is given as
    \[
    \signature(\kernel_\bx|_{[i-1, i]}) = \Phi_\Mon(\delta_i\kernel_{\bx}) = \left(1, \delta_i\kernel_{\bx}, \frac{(\delta_i\kernel_{\bx})^{\otimes 2}}{2!}, \ldots\right),
\]
where $\delta_i \kernel_\bx \coloneqq \kernel_{\bx_{i+1}}-\kernel_{\bx_i}=\kernel(\bx_{i+1},\cdot)-\kernel(\bx_i,\cdot) \in \Hil$. 
Second, using the so-called Chen identity that expresses sequence concatenation as tensor multiplication allows us to express the signature of the entire piecewise linear path as
\[
    \signature(\kernel_{\bx}) = \Phi_\Mon(\delta_1\kernel_{\bx}) \otimes \ldots \otimes \Phi_\Mon(\delta_L\kernel_{\bx}).
\]
The degree $m$ component can be explicitly expressed as
\[
    \signatureLevel{m}(\kernel_{\bx}) = \sum_{\bi: 1 \leq i_1 \leq \ldots \leq i_{m} \leq L} \frac{1}{\bi!} \delta_{i_1}\kernel_{\bx} \otimes \ldots \otimes \delta_{i_m} \kernel_{\bx}\in \Hil^{\otimes m},
\]
where for $\bi = (i_1, \ldots, i_m)$, we define $\bi! \coloneqq n_1! \cdots n_k!$ if $\bi$ contains $k = |\{i_1, \ldots, i_m\}|$ distinct elements, and $n_1, \ldots, n_k$ denotes the number of times they occur in $\bi$. Therefore, the truncated signature kernel can be written as
\begin{align*}
    \kernel_{\Sig, :M}(\bx, \by)     & = \sum_{m=0}^M \,\sum_{\substack{\bi: 1 \leq i_1 \leq \ldots \leq i_{m} \leq L\\\bj: 1 \leq j_1 \leq \ldots \leq j_{m} \leq L}} \frac{1}{\bi!\cdot \bj!} \prod_{r=1}^m \nabla_{i_r, j_r}^{\kernel}(\bx, \by),
\end{align*}
where
\[
    \nabla_{i, j}^{\kernel}(\bx, \by) \coloneqq \langle \delta_{i} \kernel_{\bx}, \delta_{j} \kernel_{\by}\rangle_H = \kernel(\bx_{i-1}, \by_{j-1}) - \kernel(\bx_i, \by_{i-1}) - \kernel(\bx_{i-1}, \by_j) + \kernel(\bx_i, \by_j).
\]
Finally, we note that we can we can apply Horner's scheme to obtain
\begin{align}\label{eq:sigkernel formula}
    \kernel_{\Sig, :M}& (\bx, \by) = \nonumber\\
    &1 + \sum_{\substack{i_1 \ge 1\\ j_1 \ge 1}} \nabla^{\kernel}_{i_1, j_1}(\bx, \by) \cdot \left( 1 + \sum_{\substack{i_2 \ge i_1\\ j_2 \ge j_1}} \frac{\nabla^{\kernel}_{i_2, j_2}(\bx, \by)}{C_2(\bi_2, \bj_2)} \cdot \left( 1+ \ldots \sum_{\substack{i_{M} \ge i_{M-1}\\ j_{M} \ge j_{M-1}}} \frac{\nabla^{\kernel}_{i_{M}, j_{M}}(\bx, \by)}{C_{M}(\bi_{M}, \bj_{M})}\right)\right),
\end{align}
where $\bi_m = (i_1, \ldots, i_m)$, $\bj_m = (j_1, \ldots, j_m)$, and $C_m(\bi_m, \bj_m) = n_{\bi_m} \cdot n_{\bj_m}$, where $n_{\bi_m} \ge 1$ is the largest number such that $i_k = i_m$ for all $k = m-n_{\bi_m}, \ldots, m$. The constants $C_m(\bi_m, \bj_m)$ are used to take into account the counting of repetitions from the $\frac{1}{\bi!\cdot \bj!}$ term.
It is straightforward to implement \eqref{eq:sigkernel formula} by a recursive algorithm which inherits the advantages of the classic Horner scheme. 
    \item[PDE Solver.]
    Another approach was given in \cite{salvi_signature_2021-1} which focuses on the case of $M=\infty$.
    The idea is to define the function $K: [0,L]^2 \rightarrow \R^d$ by the signature kernel
\[
    K(s,t) = \sigkernel(\bx|_{[0,s]}, \by|_{[0,t]})
\]
evaluated on $\bx$ restricted to $[0,s]$ and $\by$ restricted to $[0,t]$. Then, this function satisfies the following PDE
\[
    \frac{\partial^2 K}{\partial s \partial t} (s,t) = \nabla_{\lceil s \rceil, \lceil t \rceil}^{\kernel}(\bx, \by) \cdot K(s,t), \quad K(s,0) = K(0, t) = 1.
\]
Such PDEs are known as a so-called Goursat PDEs.
Hence, the computation boils down to applying a PDE solver and \cite{salvi_signature_2021-1} provide such a scheme and analyzes the error rate.
A priori, the PDE approach is focused on $M=\infty$ but cross-validating $M$ often improves performance and \cite{cass2021general} provides such a generalization of the PDE approach to approximates $\sigkernel$ for any $M \ge 1$.
\end{description}
Both algorithms work well and packages are available\footnote{\url{https://github.com/tgcsaba/KSig} and \url{https://github.com/crispitagorico/sigkernel}.} that adapt the underlying algorithms to work on GPUs.
These two algorithms are very different, even for a single evaluation of $\kernel_{\Sig,:M}(\bx,\by)$. The dynamic programming approach is exact for finite $M$ and for $M=\infty$ the approximation error stems from the factorial decay; in contrast, the PDE approach provides an approximation that is determined in by the properties of the scheme the PDE solver uses such as the step size. 
Similarly, these two algorithms are also amenable to different computational strategies when the whole Gram matrix needs to be evaluated. 
The trade-offs between these choices can vary between datasets and we simply recommend to try both algorithms.

\subsection{Hyperparameter Tuning}\label{sec:hyperparameter}
Given a fast algorithm to compute the Gram for a given parameter choice, the vast majority of the computational budget should be spent on hyperparameter tuning. 
This is true in general true for any kernel but in particular for the signature kernel where the parameter choice can drastically influcence performance.
The signature kernel has the following hyperparameters.
\begin{enumerate}
    \item \textbf{Static Kernel Parameters} $\Theta_{\kernel}$. The signature kernel takes as input a kernel $\kernel$ on $\cX$ and turns it into a kernel $\sigkernel$ on $\seq$.
Typically, the static kernel $\kernel$ is parametrized; for example, the classic RBF kernel $\kernel(\bx,\by) = \exp(\theta \|\bx-\by\|^2)$ has a hyperparamter $\theta$.
In general, hyperparameters are not limited to scalars and we denote with $\Theta_{\kernel}$ the set of possible hyperparameters.
Consequently, the signature kernel $\sigkernel$ then inherits all these hyperparameters of the static kernel $\kernel$.
\item \textbf{Truncation Parameters} $\mathbb{N}_\infty$. 
In practice, the optimal choice of the truncation level $M \in \N_\infty$ varies widely between datasets and can be as low as $M=2$. 
Informally, one expects a trade-off between the non-linearity of the static kernel $\kernel$ and the optimal truncation level $M$, but even  without kernelization the choice of optimal $M$ is subtle, \cite{Fermanian2020FunctionalLR}.
Indeed, an extreme case is to take the trivial inner product kernel and experiments show that in this case the optimal $M$ is large; in contrast for RBF and other kernels, lower choices of $M$ are typically optimal.
However, these are stylized facts and vary between datasets.
\item \textbf{Preprocessing Parameters} $\Theta_{preprocess}$. It is common practice in time-series modelling to preprocess the original time-series. 
For example, instead of the original sequence in $\R^d$ one considers a sequence in $\R^{2d}$ where the additional coordinates are given by lagging the original time series (a sliding window).
Similarly, for multi-variate time-series the different coordinates might evolve on different scales and it makes sense to normalize them; see for example \cite{toth_bayesian_2020} for some standard time-series preprocessing for the signature kernel.
In general, there is a large variety of preprocessing steps and we denote these with $\Theta_{preprocess}$.

\item \textbf{Algebraic Parameters} $\Theta_{algebra}$. Signatures can be generalized in several ways and one can replace iterated integrals by other objects. How exactly these objects are constructed goes beyond the scope of this article, but this often allows us to increase the empirical performance. 
We do not further discuss this here but refer to the ``non-geometric'' signatures in \cite{kiraly_kernels_2019,toth_seq2tens_2021} which require us to go beyond the usual (Hopf-)algebraic treatment of signatures, see \cite{diehl2020generalized}.
This induces another parameter-set $\Theta_{algebra}$.

\item \textbf{Normalization Parameters} $\Theta_{normalization}$. As we discuss in the next section, a normalization makes the signature kernel a more robust statistics; in particular, this resolves the aforementioned issues about non-compactness. 
There are different ways to normalize, and this is further discussed in~\cite{chevyrev_signature_2022}.
Furthermore, we can independently rescale the inner products for tensors of different degrees to obtain different signature kernels~\cite{cass2021general}.
We denote these normalization and rescaling parameters $\Theta_{normalization}$. 
\end{enumerate}
Therefore the hyperparameter set of the signature kernel $\sigkernel$ build from a static kernel $\kernel$ with parameter set $\Theta_{\kernel}$ is
\[
\Theta_{\sigkernel} \coloneqq  \Theta_{\kernel} \times \mathbb{N}_\infty \times \Theta_{preprocess} \times \Theta_{algebra} \times \Theta_{normalization}.
\]
We emphasize that (as for most kernels) the hyperparameter selection is essential to get strong performance. 
Depending on the problem this requires different approaches; a basic distinction is between supervised and unsupervised problems.

\begin{description}
    \item{\textbf{Supervised Learning}.}
    An objective function is given and in principle straightforward cross-validation allows us to optimize $\theta\in \Theta_{\sigkernel}$ by grid-search.
    However, since $\Theta_{\sigkernel}$ is usually high-dimensional, more sophisticated hyperparameter tuning algorithms, such as Bayesian optimization, can be very useful and are implemented in most ML toolboxes.
    Either way, standard hyperparameter tuning can be used and simple grid search typically works well. 
    \item{\textbf{Unsupervised Learning}.}
    In contrast to supervised problems the situation is less straightforward since the choice of optimization criteria is often not immediate.
    Even worse, standard heuristics such as the so-called median heuristic \cite{garreau2018large} are not justified for the signature kernel and indeed typically lead to poor performance.
    For example, we discuss this in the context of hypothesis testing in Section \ref{sec:applications}. 
    There, the median heuristic is used for some kernels, but there is no reason why this heuristic should apply to the signature kernel and indeed this would result in tests with poor power. 
    In general, for unsupervised problems the specific application must be taken into account.
\end{description}

\section{Robust Moments, Signatures, and Kernels}\label{sec:robust}
In Section~\ref{sec:good_bad_ugly}, we saw that monomial features $\Phi_\Mon$ loose the ability to approximate functions and characterize laws of random variables when we consider unbounded domains. These issues extend generically to the signature features $\signature$. While compactness can be a reasonable assumption for monomials in the case of $\cK \subset \R^d$, this is a much stronger assumption\footnote{Recall that a topological vector space is locally compact if and only if it is finite-dimensional. See also Remark \ref{rem:compact support}, that even given a  compact set that carries nearly all the support of the measure, this is not enough for characteristicness.} to make for signatures in the case of $\cK \subset \pathsOneVar$.

\begin{figure}[!htbp]
    \centering
    \includegraphics[width=0.7\textwidth]{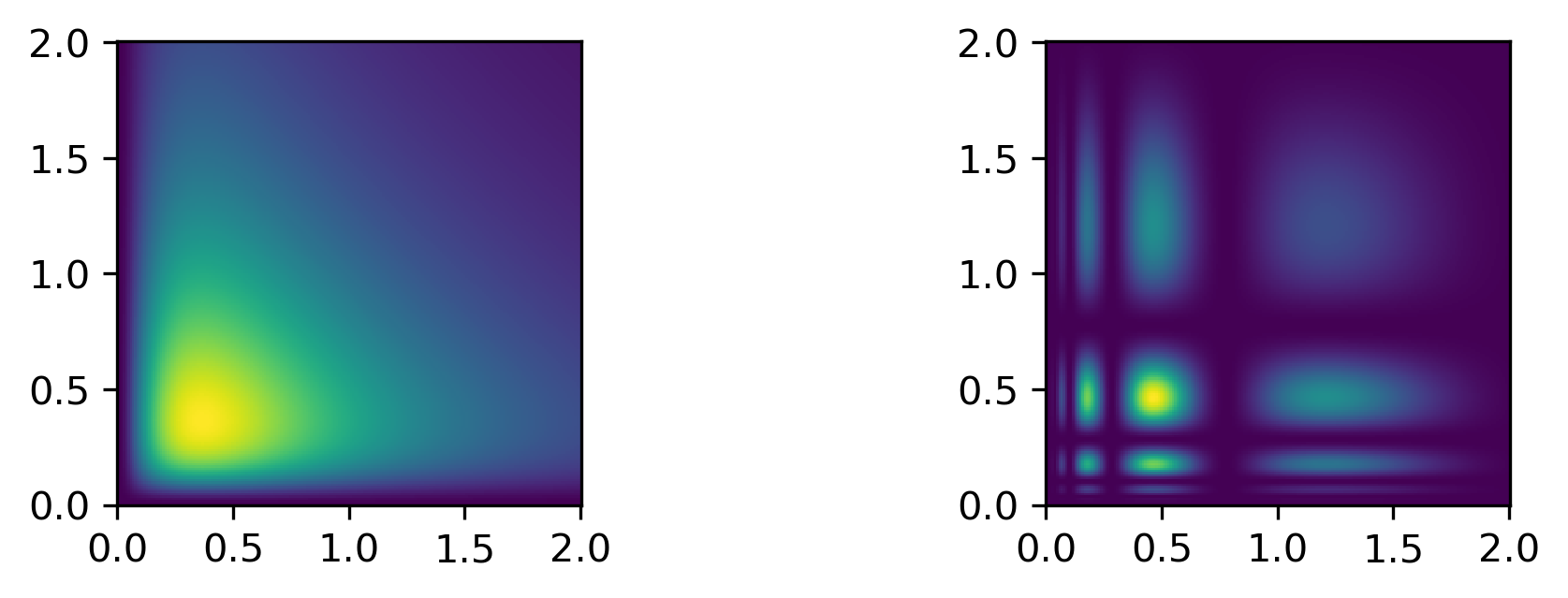}
    \caption{Probability density functions of $\bV$ (left) and $\bW$ (right).}
\end{figure}

\begin{example}[No Characteristicness] There exist two stochastic processes $\bX,\bY$ with sample paths that are straight lines, such that
\[
\E[\signature(\bX)]=\E[\signature(\bY)].
\]
For example, one can take 
\[\bX(t)=t \cdot \bV \text{ and } \bY(t) = t\cdot \bW\]
where $\bV$ is a random variable in $\R^2$ drawn from the density
\[
\bV \sim   \rho_1(\bv)=\frac{1}{2 \pi \cdot \bv_1 \bv_2} \exp\left(\frac{-(\log^2(\bv_1) + \log^2(\bv_2))}{2}\right)
\]
and $\bW$ drawn from 
\[
\rho_2(\bw_1, \bw_2) = \rho_1(\bw_1,\bw_2) \prod_{i=1}^2 \left(1 + \sin(2\pi \log(\bw_i))\right).
\]
\end{example}
One option is to make stronger model assumptions on the decay rate of moments which would allow characterizing the law; in $\R^d$ this is classic and for paths, we refer to~\cite{chevyrev2016characteristic}.
However, in an inference context we usually do not want to make strong assumptions and the above example is a warning that things can go genuinely wrong, even for trivial stochastic processes.
A different, but related, issue is the notion of robustness in statistics.
\begin{example}[No Robustness]\label{ex:non robust}
    Suppose $\bX$ is a real-valued random variable with law $\mu$.
    Fix $\epsilon \in (0,1)$ and $\bx \in \R$ and let $\bX_{\bx,\epsilon}$ be the perturbed random variable with law $\mu_{\bx,\epsilon} = (1-\epsilon) \mu + \epsilon\delta_\bx$. 
    When $\bx \gg \E[\bX]$, it is easy to see that the outlier Dirac measure $\delta_\bx$ has outsized influence on the mean $\E[\bX_{\bx,\epsilon}]$, and this influence is magnified for higher moments $\E[\bX_{\bx,\epsilon}^{\otimes m}]$. 
\end{example}
 The reason Example \ref{ex:non robust} is worrisome is that in practice we are only given a finite number of samples to estimate our statistics, such as $\E[\bX^{\otimes m}]$ or $\E[\int d\bX^{\otimes m}]$.
 In practice, this finite sample is taken from real-world observations and it is common that this data is not clean; that is, outliers such as a $\bx \gg \E[\bX]$ in the above example are not uncommon due to sensor errors, data recording, etc.\footnote{Formally, we are given samples from $\bX_{\bx,\epsilon}$ and not from $\bX$}.
 However, some statistics can still be well-estimated, and the design of such robust features has been a central topic of modern statistics \cite{hampel_robust_1986,huber_robust_2011}.
 Ironically, the standard ML benchmarks for time-series are fairly clean in terms of outliers but such outliers appear often in applications in the ``real-world''.
 \subsection{Robust Statistics}
The so-called \emph{influence function} quantifies the robustness of a statistic.
\begin{definition}[\cite{hampel_robust_1986}]
    Let $\cX$ be a topological space and $\cP(\cX)$ be the space of Borel probability measures on $\cX$. For a map $T: \cP(\cX) \rightarrow \R$, define the \emph{influence function of $T$ at $\mu \in \cP(\cX)$} by
    \[
        \IF(\bx; T, \mu) \coloneqq \lim_{\epsilon \searrow 0} \frac{T(\epsilon \delta_\bx + (1-\epsilon)\mu) - T(\mu)}{\epsilon}
    \]
    for each $x \in \cX$ where this limit exists. We say that $T$ is \emph{B-robust at $\mu$} if $\sup_{\bx \in \cX} |\IF(\bx; T, \mu)| < \infty$. 
\end{definition}
Example \ref{ex:non robust} shows that the statistic ``the $m$-th moment'', \[T(\mu)= \int \bx^{\otimes m}\, \mu(d\bx)\] is not B-robust and by the same argument we immediately see that expected signatures are not robust statistics to describe laws of stochastic processes. 
What goes wrong in both cases is that monomials, $\bx^{\otimes m}$ resp.~$\int d\bx^{\otimes m}$, explode when the argument gets big, so if we observe among our $N$ samples an extreme value $\bx$, this results in an $\epsilon=\frac{1}{N}$-weighted Dirac at $\bx$ in the resulting empirical measure and the contribution of $\bx^{\otimes m}$ weighted with mass $\epsilon\delta_\bx$, can dominate the estimator.

\subsection{Robust Signatures}
One way to address this issue is to compensate for the exponential growth as $\bx$ becomes large. 
In the case of the monomial map for $\bx \in \R^d$, this can be achieved by rescaling $\bx^{\otimes m}$ to $\lambda_m(\bx) \bx^{\otimes m}$, where $\lambda_m: \R^d \rightarrow \R$ is a function which decays quickly as $\|\bx\| \to \infty$. 
What makes the precise choice of the penalty $\lambda$ challenging is that it needs to simultaneously decay sufficiently quickly to make the statistics robust but it cannot decay too quickly in order to preserve the universal and characteristic properties. In particular, we wish to have robust analogues of Proposition \ref{prop:universal phi} and Proposition \ref{prop:universal sig}. As a technical note, in order to deal with non-locally compact spaces, we equip the space of continuous bounded functions in the next two theorems with the \emph{strict topology}. This was introduced in~\cite{giles_generalization_1971}, and further discussion on its application to robustness can be found in~\cite{chevyrev_signature_2022}. In the following, we reformulate the scaling factors $\lambda_m$ into a single normalization map $\lambda: \prod (\R^d)^{\otimes m} \to \prod (\R^d)^{\otimes m}$.

\begin{theorem}\label{thm:robust moments}
    There exists a map $\lambda: \prod_{m \ge 0}(\R^d)^{\otimes m} \to \prod_{m \ge 0} (\R^d)^{\otimes m}$ such that
    \[
    \lambda \circ\Phi_\Mon: \R^d \to \prod_{m \ge 0}(\R^d)^{\otimes m}, \quad \bx \mapsto \lambda (\Phi_\Mon(\bx))
    \]
    \begin{enumerate}
    \item is universal to $C_b(\R^d,\R)$ equipped with the strict topology, and
    \item is characteristic to the space of signed measures on $\R^d$.
    \end{enumerate}
    Moreover, $\mu\mapsto \E_{\bX \sim \mu}[\lambda\circ \Phi_\Mon(\bX)]$ is B-robust.  
\end{theorem}
In particular, the set $\{\bx \mapsto \langle \bw, \lambda \circ \Phi_\Mon(\bx) \rangle\, : \,  \bw \in \bigoplus (\R^d)^{\otimes m}\}$ is dense in $C_b(\R^d,\R)$ in the strict topology and the normalized moments $\E[\lambda \circ \Phi_\Mon(\bX)]$ always characterizes the law of $\bX$.
The existence of the function $\lambda$ is not trivial and while $\lambda$ is not given in an explicit form but it can be efficiently computed, see \cite[Proposition 14]{chevyrev_signature_2022}. We find that the same argument applies to signatures.
\begin{theorem}\label{thm: robust signatures}
    There exists a map $\lambda: \prod_{m \ge 0} \bar\Hil^{\otimes m} \to \prod_{m \ge 0} \bar\Hil^{\otimes m}$ such that
    \[
    \lambda \circ\signature: \Hil_{1-var} \to \prod_{m \ge 0} \bar\Hil^{\otimes m}, \quad \bx \mapsto \lambda (\signature(\bar \bx))
    \]
    \begin{enumerate}
    \item is universal to\footnote{Analagous to Proposition \ref{prop:universal sig}, $\bar \Hil_{1-var}$ denotes the paths of $\Hil_{1-var}$ that are augmented with a time coordinate. Explicitly, paths in $\bar \Hil_{1-var}$ are paths evolving in $\Hil \oplus \R$ of bounded 1-variation.}  $C_b(\bar \Hil_{1-var},\R)$ equipped with the strict topology, and
    \item is characteristic to the space of signed measure on $H_{1-var}$.
    \end{enumerate}
    Moreover, $\mu\mapsto \E_{\bX \sim \mu}[\lambda\circ \signature(\bar \bX)]$ is B-robust.  
\end{theorem}
We formulated Theorem \ref{thm: robust signatures} for paths evolving in Hilbert spaces in view of kernelization but the result also holds for the ordinary signature with $\kernel=\operatorname{id}$ and $\Hil=\R^d$.
Further, the analogous statements hold without adding a time coordinate by considering functions or measures on tree-like equivalence classes of paths, see \cite[Theorem 21]{chevyrev_signature_2022}.
\begin{remark}[Tensor normalization as path normalization.]
    The normalization $\lambda$ on the level of the tensors is equivalent to a normalization of the underlying path. That is $\lambda \circ \Phi_{\Sig}(\bx)$ can be rewritten as $\Phi_{\Sig}(\frac{\bx}{n(\bx)})$ for a function $n(\bx) \in \R$ of the path $\bx$.
    This way, we can think of $\lambda$ as a very specific data normalization, see \cite{chevyrev_signature_2022}.
\end{remark}
\begin{remark}[Nearly compact support does not help.]\label{rem:compact support}
    We emphasize that it can be easy to find for every $\epsilon>0$ a compact set that carries $1-\epsilon$ the mass of the probability measure. 
    The issue is that the functions we are interested, monomials or iterated integrals, behave badly on the remaining set carrying $\epsilon$ of the probability mass. 
    The normalization $\lambda$ turns this functions into bounded, hence well-behaved functions, while still keeping the necessary algebraic structure.
\end{remark}
\subsection{Robust (Signature) Kernels}
The normalization $\lambda$ allows us to overcome the drawbacks of our feature maps, $\Phi_\Mon$ and $\signature$ on non-compact sets. 
Given a static kernel $\kernel:\cX \times \cX \to \R$ with associated feature map $\varphi: \cX \rightarrow H$, we can define robust versions of the associated kernels, $\kernel_\Mon$ and $\sigkernel$, by
\[
\kernel_\Mon \coloneqq \langle \lambda \circ \Phi_\Mon(\varphi(\bx)), \lambda \circ \Phi_\Mon(\varphi(\by)) \rangle \quad \text{and} \quad \sigkernel(\bx,\by) \coloneqq \langle \lambda \circ \signature(\varphi(\bx)), \lambda \circ \signature(\varphi(\by)) \rangle.
\]
With some abuse of notion we use the same notation as the non-robust version. 
The above universality and characteristicness results, Theorem \ref{thm:robust moments} and Theorem \ref{thm: robust signatures}, immediately translate to the corresponding kernelized statement.

From the computational side, using results from~\cite{kiraly_kernels_2019, chevyrev_signature_2022}, the normalization $\lambda$ can itself be kernelized, which then gives the following.
\begin{proposition}\label{prop:robust comp}
Let $\kernel:\cX \times \cX \to \R$ be a kernel, $\bx,\by \in \seq$ be two sequences of maximal length $L$, and $M$ be the truncation level.
Then the robust truncated signature kernel $\kernel_{\Sig, :M}(\bx,\by)$ can be evaluated in $O(L^2(c_{\kernel}+M^3)+q)$ computational steps and using $O(M^2L^2+M+q)$ memory where $c_{\kernel}$ denotes the cost of a single evaluation of the static $\kernel$ and $q$ denotes the cost of finding the non-negative root of a univariate polynomial of degree $2M$. 
\end{proposition}
The algorithm in Proposition \ref{prop:robust comp} is in fact elementary given that we know how to compute the signature kernel, see \cite{chevyrev_signature_2022} for details.%

\section{Applications and Further Developments}\label{sec:applications}
Having an efficiently computable kernel for sequential data in arbitrary spaces allows us to apply kernel methods to a wide variety of problems. 
Below we discuss some of them. 
\begin{description}
\item{\textbf{A MMD for Stochastic Processes.}}
Given a set $\cX$, we often care about the space of probability measures $\cP(\cX)$ on this space. 
We have encountered the notion of characteristicness which says that we can uniquely represent a measure $\mu \in \cP(\cX)$ as an expectation, that is
\[
\mu \mapsto \E_{\bX \sim \mu}[\varphi(\bX)]
\]
is injective. 
If the co-domain of $\varphi$, the feature space, has a norm, this would induce metric 
\[
d(\mu,\nu) = \|\E_{\bX \sim \mu}[\varphi(\bX)]-\E_{\bY \sim \nu}[\varphi(\bY)]\|.
\]
The caveat is that these expectations take values in an infinite-dimensional space. 
However, for the choice $\varphi(\bx)=\kernel(\bx,\cdot)$ it turns out that the above can be efficiently computed; a straightforward computation using that $\Hil$ is a RKHS shows that
\begin{align*}
\|\E[\varphi(\bX)]-\E[\varphi(\bY)]\|^2  &= \langle \E[\varphi(\bX)]-\E[\varphi(\bY)],\E[\varphi(\bX)]-\E[\varphi(\bY)]  \rangle\\
&=\langle \E[\varphi(\bX)],\E[\varphi(\bX)]\rangle-2 \langle \E[\varphi(\bX)],\E[\varphi(\bY)]\rangle +\langle \E[\varphi(\bY)], \E[\varphi(\bY)]  \rangle \nonumber \\
&=\E[\kernel(\bX,\bX')] -2\E[\kernel(\bX,\bY)] +\E[\kernel(\bY,\bY')],
\end{align*}
where $\bX, \bX' \sim \mu$ resp.~$\bY, \bY' \sim \nu$ are independent.
The right-hand side is straightforward to estimate when only finite samples of $\mu$ and $\nu$ are given; in particular,
\[
d_{\kernel}(\mu,\nu) = \sqrt{\E[\kernel(\bX,\bX')] -2\E[\kernel(\bX,\bY)] +\E[\kernel(\bY,\bY')]}
\]
is a metric which can be easily estimated with finite samples. 
The above metric is also known as kernel maximum mean discrepancy (MMD) since again a direct calculation shows that 
\[
d_{\kernel}(\mu,\nu)=\sup_{f \in \Hil: \|f\| \le 1} \left| \int_\cX f(\bx) \mu(d\bx) - \int_\cX f(\bx) \nu(d\bx)\right|.
\]
Hence, $d_{\kernel}$ is the maximum difference we can achieve by testing our measure against functions that lie in the unit ball of the RKHS of $\kernel$.
MMDs have found many applications in ML, see \cite{muandet2017kernel} for a survey. 
Having the signature kernel at hand, then immediately yields a MMD for probability measures on paths; in particular, we obtain a metric for laws of stochastic processes. 
The theoretical guarantees of the signature kernel then translate into theoretical guarantees for the signature kernel MMD $d_{\sigkernel}$.
However, the topology on the space of probability measures induced by the signature MMD is more subtle.
When the RKHS is finite-dimensional it induces weak convergence but not for infinite-dimensional RKHS; see~\cite{chevyrev_signature_2022} for details.

\item{\textbf{Hypothesis Testing}.}
Among the most popular hypothesis testing methods is two-sample testing and signature kernels allow us to do this with stochastic processes. In particular, we can test
\[
H_0: \operatorname{Law(\bX)}= \operatorname{Law}(\bY) \text{ against }H_1:  \operatorname{Law(\bX)} \neq \operatorname{Law}(\bY)
\]
where $\bX=(\bX_t)$ and $\bY=(\bY_t)$ are two stochastic processes.
The data we are given we are given are $m$ sample trajectories from $\bX$ and $n$ sample trajectories from $\bY$.
Such situations are not uncommon, for example consider a cohort of patients in a medical trial who are divided into two subgroups (e.g.~one taking medication and the other not) and we measure medical markers during the trial period. For another example, see \cite{andres2022signature} for an application of the signature MMD to test and validate hypothesis for real-world economic scenarios.
Following \cite{gretton_kernel_2012} we reject $H_0$ if
\begin{align}\label{eq:test}
d_{\sigkernel}(\hat \mu_X, \hat \mu_Y) > c(\alpha,m,n).
\end{align}
Here, $\hat \mu_X = \sum_{i=1}^m \delta_{\bx_i}$ and $\hat \mu_Y = \sum_{j=1}^n \delta_{\by_j}$ are the empirical measures given by our observed sequences\footnote{With slight abuse of notation, each $\bx_i \in \seq$ denotes a whole sequence and not single sequence entry as before.} $\bx_1,\ldots,\bx_n,\by_1,\ldots,\by_m \in \paths$ and $c(\alpha,m,n)$ is a constant such that the test \eqref{eq:test} has a significance level\footnote{Recall that the type I error of a test is the probability of falsely rejecting $H_0$ and the type II error is falsely accepting $H_0$. If the type I error is bounded by $\alpha$ then one says the test has power $\alpha$.} $\alpha$.
In \cite{gretton_kernel_2012}, several estimators are proposed for $d_{\sigkernel}(\hat \mu_X, \hat \mu_Y)$ and these estimators differ in their degree of bias and their computational complexity. 
The function $c$ is determined by asymptotic distribution in \cite{gretton_kernel_2012}, but this choice leads in general to a very conservative test criteria, that is one seldom rejects the $H_0$. 
In practice, permutation tests often result in better empirical performance for MMD hypothesis testing, and this applies especially to the signature kernel; we refer to \cite{chevyrev_signature_2022} for details and for experiments and benchmarks against other hypothesis tests.
Beyond testing the equivalence of laws, kernel learning allows us to test for independence between the coordinates of a random variable \cite{gretton2007kernel}. 
Note that for stochastic processes, independence between coordinates is challenging since statistical dependencies can manifest between different times in different coordinates. However, quantifying and testing independence for stochastic processes have important applications such as Independent Component Analysis for stochastic processes \cite{schell_nonlinear_2021,schell2023robustness,bonnier2020signature}.

Finally, note that the signature kernel $\sigkernel$, and hence the MMD $d_{\sigkernel}$, is parametrized by hyperparameters $\theta_{\sigkernel} \in \Theta_{\sigkernel}$ and already for simple supervised learning, the hyperparameter choice was essential for good performance.
Using a standard grid-search to choose the best $\theta$ is no longer applicable for hypothesis testing since using the given data would introduce correlations which invalidate the test. 
This is a generic problem for kernel MMDs in testing problems and several solutions have been proposed which optimize for the test power. 
What works in well in practice is to follow \cite{Sriperumbudur2009KernelCA} which proposes to replace $d_{\sigkernel}$ by 
\[
(\bx,\by) \mapsto \sup_{\theta \in \Theta} d_{\sigkernel(\theta)}(\bx,\by)
\]
where we write $\sigkernel(\theta)$ to denote the dependence of $\sigkernel$ on $\theta$.
The above is again a metric and while it is not the optimal choice for the test power, it guarantees a sample test with given significance; moreover, it works well in practice with the signature kernel, see \cite{chevyrev_signature_2022} for details.
A related issue, is the choice of the static kernel $\kernel$ itself: for essentially all previous applications of the signature kernel, using the Euclidean inner product kernel as static kernel is outperformed by non-linear static kernels such as RBF.
However for hypothesis testing, using the standard inner product as the static kernel for the signature kernel can be competitive with the signature kernel built over non-linear static kernels for some datasets. 
We refer again to \cite{chevyrev_signature_2022} for a discussion of this and its relation to the question of hyperparameter choice.

\item{\textbf{Signature MMDs in Generative Models.}}
The signature MMD is a metric between any two measures on path space and can be computed efficiently. 
In particular, it does not make any parametric assumptions on the two measures that are compared, hence it can be used in a generative model to decide whether the generator produces samples that are close enough to a given target distribution. 
This has been applied to generative adversarial networks (GANs) to model financial data \cite{buehler2020data} and to generative models for (neural) stochastic differential equations \cite{kidger2021neural}.
Again, the choice of hyperparameter and static kernel is challenging from the theoretical perspective but this is not specific to the signature kernel and general heuristics from the ML literature might apply.

\item{\textbf{Gaussian Processes with Signature Covariances.}} We often require not only point predictions, but also estimates of the associated uncertainties.
Bayesian methods provide a structured approach to update and quantify uncertainties and in particular, Gaussian processes are a powerful and non-parametric method which provides a flexible way to put priors on functions of the data; see~\cite{rasmussen_gaussian_2005}.
Formally, a Gaussian process on $\cX$ is a stochastic process $f=(f_\bx)_{\bx \in \cX}$ indexed by $\bx \in \cX$ such that $f \sim N(m,\kernel)$ is determined by a mean and a covariance function
\[
m:\cX \to \R \text{ and a positive definite } \kernel:\cX \times \cX \to \R
\]
where $f_\bx$ is Gaussian with
\[
\E[f_\bx] = m(\bx) \text{ and } \operatorname{Cov}(f_\bx,f_\by).
\]
Note that we break here with our convention to denote random variables with capital letters and write $f$ instead of $F$ to be consistent with the GP literature.
Without loss of generality, one can take $m(\bx)=0$ and consider data in an arbitrary space $\cX$ provided a kernel $\kernel$ is given. 
If the kernel $\kernel$ has nice properties these often translate into desirable properties of the GP. 
Given the signature kernel, one can immediately consider a GP for sequential data. 
However, two issues arise, one more theoretical and one more applied
\begin{itemize}

    \item The existence of Gaussian processes with continuous sample trajectories is not trivial on non-compact spaces; see for example \cite{Adler2007RandomFA} for simple counterexamples.
    The required regularity estimates to ensure existence can be derived by using a classical result of Dudley \cite{dudley2010sample} and bounding the growth of the covering number which requires us to utilize the structure of $\kernel$ and $\cX$.
    In \cite{toth_bayesian_2020}, such covering number estimates are computed for the signature kernel which then gives the existence of a GP indexed by paths; further, it yields a modulus of continuity for the GP sample paths $\bx \mapsto f_\bx$.

   \item Plain vanilla GP inference requires to invert large Gram matrices which scales with complexity $O(n^3)$ in the number of samples $n$. 
   For most datasets, this is too costly and several approximation have been developed. 
   Arguably the most popular are variational inference and inducing point methods.
   To efficiently apply these to the signature kernel requires using the structure of tensors. 
   The article \cite{toth_bayesian_2020} introduced GP with signature covariances and developed a variational approach based on sparse tensors; and \cite{lemercier2021siggpde} developed a variational approach that uses diagonal approximations to the Gram matrix which allows for quick matrix inversion.
\end{itemize}
\item{\textbf{Strings, Time Series, and Classic Sequence Kernels.}}
Kernel learning for sequences is a classic topic and the case when the state space is finite or discrete -- in our notation $\cX$ is a finite or countable set of letters and $\seq$ is the set of strings with letters from $\cX$ -- has received much attention.
Haussler developed a relation-convolution kernel framework \cite{haussler1999convolution} that gave rise to so-called string kernels; see for example \cite{Lodhi2002TextCU,Leslie2004FastSK} for algorithms and applications in text learning and biology.
These approaches have been extended to allow for general state spaces with the alignment kernel \cite{bahlmann2002online,shimodaira2001dynamic,Cuturi2006AKF} that search over all possible alignments (``time-warpings''), although this may not result in a positive definite kernel. 
However, both the classic string kernel and the global alignment kernel, arise as special cases or minor variations of the signature kernel construction and other variations are possible, see \cite[Section 5 and Remark 4.10]{kiraly_kernels_2019} and \cite[Appendix B.3]{toth_seq2tens_2021}, thereby providing theoretical guarantees for these kernels.
\item{\textbf{Time-Parametrization (In-)variance.}}
The signature $\bx \mapsto \Phi_{\Sig}(\bx)$ is injective up to translation of the starting point $\bx(0)$ of the path $\bx=(\bx(t))_{t \in [0,T]}$ and so-called ``tree-like'' equivalence \cite{hambly_uniqueness_2010}.
The latter means that paths are equivalent up to reparametrization and backtracking.
In particular, $\Phi_{\Sig}(\bx)=\Phi_{\Sig}(\bx^\rho)$ where $\rho:[0,S]\to [0,T]$ is non-decreasing and $\bx^\rho(t)\coloneqq \bx (\rho(t))$.
Adding a time coordinate as in Section \ref{sec:universal}, $\bar \bx(t)\coloneqq(t,\bx(t))$, destroys any tree-like equivalence, that is $\bx \mapsto \Phi_{\Sig}(\bar \bx)$ is injective on $\pathsOneVar$, hence distinguishes paths that differ by a time-change.
Removing such a big class of invariances (the set of time-changes is infinite-dimensional) is a powerful tool and has received much attention in engineering \cite{sakoe1978dynamic} where this is known as ``time-warping''; in the above discussed discrete case of strings, this is known as global alignment \cite{kruskal1983overview}.
The same theoretical guarantees discussed in Section \ref{sec:universal} translate to the corresponding statements about functions or probability measures on such equivalence classes.

\item{\textbf{Weak and Adapted Topologies.}}The signature MMD is a metric on the space of probability measures on either $\seq$ or $\paths$.
Arguably the most important topology on the space of probability measures is the weak topology and we discussed above when the signature resp.~signature MMD metrizes this metric. 
However, the weak topology completely ignores the filtration of a stochastic process. The need for finer topologies for laws of stochastic processes that take the filtration into account has been understood a long time ago \cite{Hoover1984AdaptedPD} and have recently seen a revived interest \cite{pammer2022note}. These finer topologies run under the name of higher rank topologies (the higher the rank, the more of the filtration structure is taken into account).
In \cite{bonnier_adapted_2020}, higher rank signatures are introduced which take filtrations into account and allow us to metrize these higher rank topologies. 
The same methodology that constructs the signature kernel given a static kernel can be carried out by using the higher-rank signatures and \cite{salvi_higher_2021,horvath2023optimal} use the PDE solver to estimate the associated MMD in various machine learning tasks.

\item{\textbf{Rough Paths.}}
Throughout we made the assumption that our paths are of bounded variation, that is we defined $\sigkernel$ on $\pathsOneVar$. 
However, many stochastic processes have sample paths that are of unbounded variation, and thus are not elements of $\pathsOneVar$.
This is fairly straightforward to address with tools from stochastic analysis and rough path theory.
There are two ways to think about this: 
\begin{enumerate}[label=(\roman*)]
\item \label{itm:embed} In practice we are given sequences and not paths, and these can always be embedded as bounded variation path, $\seq \hookrightarrow \paths$. 
\item Even though we can embed any sequence as a bounded variation path as in \ref{itm:embed}, this can become flawed in the high-frequency scaling limit since the iterated integrals might not converge.
A more radical approach is to define the signature kernel for so-called rough paths. 
This requires us to prescribe the first iterated integrals since these are not well-defined as Riemann--Stieltjes integrals.
In \cite[Appendix B]{kiraly_kernels_2019} this was done for the so-called geometric rough paths, which are those for which piecewise linear approximations converge\footnote{In the notation of \cite[Appendix B]{kiraly_kernels_2019}, $\kernel^{+}_{(d,m)}$ with $d=\lfloor p \rfloor$ where $p$ denotes the $p$-variation of the rough path. 
In our current notation $d,m$ are parts of the parameter $\theta_{algebra} \in \Theta_{algebra}$ of $\sigkernel$.} and hence allow for canonical choices of the higher iterated integrals. 
This covers semimartingales and large classes of other popular stochastic processes.
The same algorithm applies if other choices of the higher order integrals are prescribed\footnote{For example, to describe a "rough" path in $d$ dimensions, one could consider sequences $\bx=(\bx_0,\ldots,\bx_L)$ in the state space $\cX=G_{m,d}$, the group-like elements; informally $\bx_i$ are the first $m$ "iterated integrals" of the underlying path on the interval $[t_i,t_{i+1}]$. In the notation of Section \ref{sec:hyperparameter}, producing such group-like elements from raw sequential data could be regarded as a preprocessing step.} but these are typically not available in the usual sequence datasets.
\end{enumerate}

\item{\textbf{Beyond (Geometric) Rough Paths.}}
Throughout we focused on the classical signature but for many applications so-called non-geometric or branched rough paths provide a different way to describe a path as a sequence of tensors. 
Much of the discussion we had so far can be translated to branched rough paths with minor modification, that is one can again define a "signature kernel" between branched rough paths\footnote{This is in the ArXiv version v1 of \cite[Appendix D]{arXivchevyrev2022signature}. 
Appendix D is not part of the published version \cite{chevyrev_signature_2022} of \cite{arXivchevyrev2022signature} for brevity.}.
One can even go beyond that and argue that the essential structure is that of embedding a sequence into a non-commutative algebra in which we "stitch" the sequence elements together by the non-commutative multiplication in the algebra. 
That this multiplication is non-commutative is essential since this algebraically captures the sequence order. 
From this point of view, the signature and kernel that we discussed so far realizes this abstract method by instantiating it with the so-called shuffle Hopf algebra and using as embedding the Lie group exponential; branched rough paths replace the shuffle algebra by the Connes--Kreimer algebra.
Following this approach, we can simply parametrize the algebraic structure and learn the best choice from the data.
The underlying theory is beyond the current scope but we refer to \cite{toth_seq2tens_2021} for a demonstration of how parametrizing the underlying algebraic structure and going beyond geometric rough paths can improve benchmark performance, often significantly, and how to deal with the computational challenges.
Similarly, see the discussion about the non-geometric signature kernel\footnote{Denoted as $\kernel^+_{(d,m)}$ when $d\neq m$ in the notation of~\cite{kiraly_kernels_2019}.} in \cite[Remark 4.5]{kiraly_kernels_2019}.
Besides improving performance, such variations lead to interesting and challenging algebraic question which no longer can be treated within a Hopf algebraic framework.
We refer to \cite{diehl2020generalized} for a detailed study of the underlying algebraic structures.

\item{\textbf{Evolving (Non-Euclidean) Data.}}
The signature kernel transforms any static kernel on the domain $\cX$ into a kernel on the corresponding space of sequences $\seq$, allowing us to design kernels for objects evolving in general (non-Euclidean) spaces.
One example arises from the field of topological data analysis, in which a method called persistent homology allows us to quantify the multi-scale topological structure` of data such as point clouds in highly nonlinear objects called persistence diagrams~\cite{dey_computational_2022}. 
In \cite{Chevyrev2018PersistencePA}, static persistence diagrams are transformed into paths, and the signature kernel allows us to capture the topological properties of a static point cloud.
Given an evolving point cloud, for instance representing a swarm of agents, we obtain a sequence of persistence diagrams which summarize the dynamic topological structure of the data, and signature kernels again provide an effective method to perform learning tasks such as regression~\cite{giusti_signatures_2021}.
Another example is based on the observation is that the definition of the signature only requires an inner product on the derivatives of a sequence. For paths evolving on Lie groups, where derivatives lie in the associated Lie algebra, one can formulate the signature kernel in exactly the same way~\cite{lee_path_2020}.
Finally, another case where a general approach to sequences can be studied with the signature kernel is to analyze classic deep learning architectures via the attractive theoretical properties of the signature kernel; for example, \cite{fermanian2021framing} studies RNNs in terms of the signature kernel, and \cite{cirone2023neural} study scaling of Resnets from this perspective.
\item{\textbf{Likelihood-free Inference.}}
Given a parametrized statistical model, the likelihood function captures its fit to data.
Often the underlying models are complex which leads to intractable likelihood functions and makes it infeasible to find the optimal parameter.
So called likelihood-free inference addresses these challenges and \cite{Dyer2022AmortisedLI} combine the signature kernel and its MMD with so-called density ratio estimation to develop a likelihood-free inference approach; see also \cite{Dyer2022ApproximateBC} for its use in so-called approximate Bayesian computation.
\end{description}

\section*{Acknowledgements}
HO and DL were supported by the Hong Kong Innovation and Technology Commission (InnoHK Project CIMDA). HO was supported by the EPSRC [grant number EP/S026347/1].

\bibliographystyle{plain}
\bibliography{signatures_survey}
\end{document}